\documentclass[12pt]{amsart}
\pdfoutput=1

\usepackage[headings]{fullpage}

\usepackage{graphicx}
\usepackage{color}
\usepackage{url}
\usepackage[bf]{caption}
\usepackage[all]{xy}
\usepackage{type1cm}
\usepackage{enumerate}
\usepackage[hidelinks]{hyperref}

\newtheorem{definition}{Definition}[section]
\newtheorem{theorem}[definition]{Theorem}

\newtheorem{lemma}[definition]{Lemma}

\newtheorem{cor}[definition]{Corollary}
\theoremstyle{definition}

\theoremstyle{definition}
\newtheorem{remark}[definition]{Remark}

\def\C{\mathbb{C}}
\def\Q{\mathbb{Q}}
\def\H{\mathbb{H}}

\def\Z{\mathbb{Z}}
\def\R{\mathbb{R}}
\def\calO{\mathcal{O}}

\def\SL{\mathrm{SL}}
\def\PSL{\mathrm{PSL}}

\def\PGL{\mathrm{PGL}}

\def\pqr{\{p,q,r\}}

\newcommand\groupMarker[1]{\mathrm{#1}}
\newcommand\mfdMarker[1]{\mathrm{\mathcal{#1}}}

\newcommand\tetGroup[1]{\groupMarker{\Gamma}^{#1}}

\newcommand\univPrinCong[2]{\groupMarker{U}^{#1}_{#2}}
\newcommand\MunivPrinCong[2]{\mfdMarker{U}^{#1}_{#2}}
\newcommand\pslPrinCongBig[2]{\groupMarker{W}^{#1}_{#2}}
\newcommand\MpslPrinCongBig[2]{\mfdMarker{W}^{#1}_{#2}}
\newcommand\pslPrinCong[2]{\groupMarker{X}^{#1}_{#2}}
\newcommand\MpslPrinCong[2]{\mfdMarker{X}^{#1}_{#2}}
\newcommand\pglPrinCong[2]{\groupMarker{Y}^{#1}_{#2}}
\newcommand\MpglPrinCong[2]{\mfdMarker{Y}^{#1}_{#2}}
\newcommand\exceptionalObject[2]{\groupMarker{Z}} 
\newcommand\MexceptionalObject[2]{\mfdMarker{Z}} 
\newcommand\mfd{\groupMarker{M}}
\newcommand\Mmfd{\mfdMarker{M}}

\newcommand\barMunivPrinCong[2]{\bar{\mfdMarker{U}}^{#1}_{#2}}

\newcommand\myComment[1]{}

\begin{document}

\title{Regular Tessellation Link Complements}
\author{Matthias G\"orner}

\begin{abstract}
By a regular tessellation, we mean any hyperbolic 3-manifold tessellated by ideal Platonic solids such that the symmetry group acts transitively on oriented flags. A regular tessellation has an invariant we call the cusp modulus. For small cusp modulus, we classify all regular tessellations. For large cusp modulus, we prove that a regular tessellations has to be infinite volume if its fundamental group is generated by peripheral curves only. This shows that there are at least 19 and at most 21 link complements that are regular tessellations (computer experiments suggest that at least one of the two remaining cases likely fails to be a link complement, but so far, we have no proof). In particular, we complete the classification of all principal congruence link complements given in Baker and Reid for the cases of discriminant $D=-3$ and $D=-4$. We only describe the manifolds arising as complements of links here, with a future publication ``Regular Tessellation Links'' giving explicit pictures of these links.
\end{abstract}

\maketitle

\tableofcontents

\section{Introduction}

According to \cite{AitchisonRubstein:CombCubings}, tessellations by Platonic solids ``have played a significant role in the exploration and exposition of 3-dimensional geometries and topology,'' and we refer the reader to Aitchison and Rubinstein's excellent introduction for examples. Existing literature has investigated the tessellations by ideal (and thus hyperbolic) Platonic solids that can occur as knot and link complements. Following \cite{coxeter:honeycombsHyp,coxeter:RegPolytopes}, we know that there is an ideal tessellation of hyperbolic 3-space $\H^3$ by the regular ideal tetrahedron, octahedron, cube, and dodecahedron, but not the icosahedron. It is known that each of the first four Platonic solids also tessellates some knot or link complement. An example for the tetrahedron is the figure-eight knot, and for the octahedron the Borromean rings \cite{thurstonLevy:ThreeDimGeomTop}. Aitchison and Rubinstein completed this work by constructing cubical links and dodecahedral knots \cite{AitchisonRubstein:CombCubings}. Note that for some solids, the literature gives knots, but for others, only links. In fact, the figure-eight knot and the two dodecahedral knots are the only three knots whose complements are tessellated by an ideal Platonic solid. This follows from Reid's result that the figure-eight knot is the unique arithmetic knot \cite{reid:arithmeticity} and Hoffman's analysis of the nonarithmetic dodecahedral case \cite{hoffmann:dodecahedral}.

However, there are infinitely many link complements tessellated by ideal Platonic solids. To see this, pick some link that has an unknotted component and whose complement is tessellated by an ideal Platonic solid, for example, the Borromean rings. Then, construct an arbitrarily large cyclic cover of the link branched over the chosen unknotted component. If we require as natural symmetry condition that the link complement be a regular ideal tessellation, then the number becomes finite again, and it is a natural question and the goal of this paper to classify them. Note that all manifolds considered here are orientable but that we allow chiral regular tessellations of a manifold, so a tessellation is defined to be regular if each flag consisting of a solid, an adjacent face and an edge adjacent to the face can be taken to any other flag through an isometry. When we say ``regular tessellation'' we mean any hyperbolic manifold together with a tessellation fulfilling this condition. This generalizes the traditional use of that term that assumed that the underlying manifold is $\mathbb{S}^3$, $\mathbb{E}^3$, or $\H^3$. For the tetrahedron, a tessellation is a regular if and only if the underlying hyperbolic manifold is ``maximally symmetric'' in the sense that no other manifold has more orientation-preserving symmetries per volume (this follows from \cite{Meyerhoff:MinVolCuspedOrbifold} where the minimum volume orbifold was determined).

\subsection{Examples of Known Regular Tessellations}

Ideal regular tessellations by Platonic solids arise in various contexts. If we allow orbifolds for a moment, there is a particularly easy example: the boundary of a 4-simplex is a regular tessellation by 5 regular ideal hyperbolic tetrahedra, as shown on the left in Figure~\ref{fig:5chainLink}. To obtain a regular tessellation of a cusped hyperbolic manifold, take the unique manifold double-cover of the orbifold. The result is the complement of the minimally twisted 5-component chain link L10n113, which conjecturally has the smallest volume among the 5-cusped orientable hyperbolic manifolds. Its Dehn fillings have been extensively studied: the exceptional ones were recently classified in \cite{MPR:fivechainfilling}, and the remaining Dehn fillings yield almost every manifold of the cusped Callahan-Hildebrand-Weeks census \cite{CallahanHildebrandWeek:cuspedCensus} and thus also of the closed Hodgson-Weeks census \cite{HodgsonWeeks:closedCensus}, a remarkable property also featured in the experiments of \cite{dunfield:virthaken} on the virtual Haken conjecture.

\begin{figure}[h]
\begin{center}
\scalebox{0.88}{\input{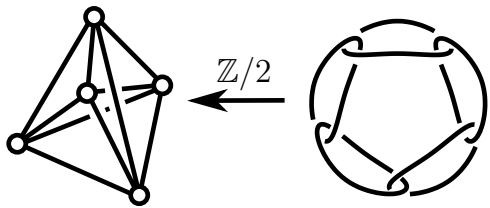}}
\caption{\label{fig:5chainLink}The orbifold on the left (``pentangle'') is tessellated by 5 regular ideal hyperbolic tetrahedra, 3 meeting at an edge (edges show order-2 singular locus, dots orbifold cusps, and note that the 5 tetrahedra are the faces of the boundary of a 4-simplex that become ideal when removing the 0-skeleton). The orbifold is double-covered by a manifold that is the complement of the minimally twisted 5-component chain link and a regular tessellation by 10 ideal hyperbolic tetrahedra \cite{dunfield:virthaken}. In the notation introduced later, it is $\MunivPrinCong{\{3,3,6\}}{2}$.}
\end{center}
\end{figure}

Another example of a regular tessellation with an easy combinatorial description comes from chessboard complexes defined in \cite{ziegler:shellChess, BjornerLovaszVrecicaZivaljevic:Chessboard}:
\begin{definition}
The $m\times n$ {\it chessboard complex} is a simplicial complex consisting of a $k$-simplex for every non-taking configuration of $k$ rooks (``that is, no two rooks on the same row or column'') on the $m\times n$ chessboard. A face of a $k$-simplex is identified with the $k-1$-simplex corresponding to the respective subset of $k-1$ rooks.
\end{definition}
The $4\times 5$ chessboard complex happens to yield a regular tessellation by 120 hyperbolic ideal tetrahedra \cite{eppstein:mathoverflow} (namely $\MunivPrinCong{\{3,3,6\}}{2+2\zeta}$). To see this, look at the $2\times 3$ and the $3\times 4$ chessboard complex: The $2\times 3$ chessboard complex is a 6-cycle. The 6-cycle is the link of a vertex in the $3\times 4$ chessboard complex that is a torus with 24 triangles. This torus is a link of a vertex in the $4\times 5$ chessboard complex. Hence, there are 6 tetrahedra meeting at an edge and removing the vertices yields a 3-manifold with toroidal cusps.

The last example we give is the Thurston congruence link, an 8-component link whose complement admits a regular tessellation by 28 hyperbolic ideal tetrahedra with a fascinating connection to the Klein quartic, the complex projective curve defined by $x^3y+y^3z+z^3x=0$. In \cite{thurston:HowToSee}, Thurston noticed that the cusp neighborhoods in the complement of the minimally twisted 7-component chain link appear to be close to a regular pattern and that drilling a geodesic ``crystallizes'' the hyperbolic manifold so that it admits a regular tessellation by 28 ideal tetrahedra (namely $\MunivPrinCong{\{3,3,6\}}{2+\zeta}$). He gave a picture of the resulting 8-component link and noticed a fascinating connection with the Klein quartic: they share the same orientation-preserving symmetry group $\PSL(2,7)$, the unique simple group of order 168. In fact, the 2-skeleton of the above regular tessellation forms an immersed punctured Klein quartic \cite{agol:CongruenceLink}.

\subsection{Arithmetic and Congruence Links}

The Thurston congruence link is also an example of a congruence link and thus illustrates the connection of the classification result here to arithmetic and congruence links. Motivated by Thurston's question 19 in \cite{thurston:bulletinKleinianGroups} noting the ``special beauty'' of these manifolds, Reid and Baker have been classifying small congruence links, e.g., \cite{bakerReid:prinCong}. For a subclass of congruence links, namely the principal congruence links for discriminant $D=-3$ and $D=-4$, the main theorem here allows a complete classification. Let $z\in\calO_D$ where $\calO_D$ is the ring of integers in the imaginary quadratic number field $\Q(\sqrt{D})$. Recall that the principal congruence manifold (denoted here by $\MpslPrinCongBig{D}{z}$, see Section~\ref{sec:FinUnivRegTess}) of level $z$ is the quotient of $\H^3$ by all matrices in $\PSL(2,\calO_D)$ that are congruent to the identity matrix modulo the ideal $\langle z\rangle$. For discriminant $D=-3$ and $D=-4$, a principal congruence manifold admits a regular tessellation by ideal tetrahedra, respectively, octahedra. Thus, we can apply the main theorem to list all $\MpslPrinCongBig{-3}{z}$ and $\MpslPrinCongBig{-4}{z}$ that are link complements in Corollary~\ref{cor:classPrinCong}.

\subsection{Remarks}

It should be remarked that although the Borromean rings and the Whitehead link can be tessellated by octahedra and the alternating 4-component chain link $8^4_1$ by cubes, these tessellations are not regular as the smallest regular tessellation link (smallest in terms of both volume and number of components) has 5 components. In particular, no knot complement is a regular tessellation.

Furthermore, note that there is a subtle difference between classifying regular tessellation link complements and regular tessellation links. Whereas the number of regular tessellation link complements is finite, the number of regular tessellation links is, a priori, not finite as a link is in general not determined by its complement (as pointed out in \cite{bakerReid:prinCong}). In fact, twisting along a disc spanned by a component of the link in Figure~\ref{fig:5chainLink} yields an infinite family of regular tessellation links with the same complement. Thus, we only classify link complements here, not links.

\section{Main Theorem and Overview}

Consider the regular ideal tessellations $\pqr$ of $\H^3$ by ideal Platonic solids (see~\cite{AitchisonRubstein:CombCubings}): $\{3, 3, 6\}$ tessellated by tetrahedra, $\{3, 4, 4\}$ by octahedra, $\{4, 3, 6\}$ by cubes, and $\{5, 3, 6\}$ by dodecahedra. We use the upper half space model of $\H^3$ and identify $\partial \H^3$ with $\C\cup\{\infty\}$ so that $\PSL(2,\C)$ and $\PGL(2,\C)$ act by M\"obius transformations and $\PSL(2,\C)\cong \PGL(2,\C)\cong \mathrm{Aut}^+(\H^3)$. For normalization, we move each of the above regular ideal tessellations in $\H^3$ such that there is a face with three consecutive vertices at the points $\infty$, $0$, and $1$. Let $\tetGroup{\pqr}\subset\mathrm{Aut}^+(\H^3)$ be the orientation-preserving isometries of such a regular tessellation of $\H^3$.

\begin{definition}\label{def:algDefRegTess}
A manifold $\Mmfd$ is a regular tessellation of type $\{p,q,r\}$ if it is a quotient $\Mmfd=\H^3/M$ by a torsion-free normal subgroup $M$ of $\tetGroup{\pqr}$.
\end{definition}

 Let $z$ be a number of the form $a+bu$ where $u=e^{2\pi i/r}$ and $a, b\in\Z$. For $r=4$, we write $z=a+bi$ and, for $r=6$, $z=a+b\zeta$ where $\zeta=\frac{1+\sqrt{-3}}{2}$.

\begin{definition} \label{def:univRegTess}
The universal regular tessellation of cusp modulus $z$ is the quotient $$\MunivPrinCong{\pqr}{z}=\frac{\H^3}{\univPrinCong{\pqr}{z}}$$ where $\univPrinCong{\pqr}{z}$ denotes the normal closure of $p_z=\left(\begin{array}{cc}1 & z\\ 0 & 1\end{array}\right)$ in $\tetGroup{\pqr}$, i.e., the smallest normal subgroup of $\tetGroup{\pqr}$ containing $p_z$.
\end{definition}

 The groups $\univPrinCong{\{3,q,r\}}{z}$ were also called ``stabilizer of infinity in the principal congruence subgroup'' in \cite{bakerReid:prinCong}.
The quotient space can actually be an orbifold. It can also be infinite-volume. We explain the regular tessellation structure and the universal property of $\MunivPrinCong{\pqr}{z}$ in detail in Section~\ref{sec:regTessSolids} and only remark here that multiplication of $z$ by $u$ leaves the universal regular tessellation unchanged and complex conjugation only reflects it.
Thus, for classification purposes, we only have to enumerate those $z$ that are in canonical form defined as follows:
\begin{definition}
Let $z=a+bu$. If $a\geq b\geq 0$, we say that $z$ is in {\bf canonical form}.
\end{definition}

We can now state the main theorem:

\begin{theorem}\label{thm:main}
If $\Mmfd$ is a link complement admitting a regular tessellation, then it must be a finite-volume manifold universal regular tessellation. All finite-volume universal regular tessellations $\MunivPrinCong{\pqr}{z}$ are listed in Table~\ref{table:finiteVolUnivReg}, so there are 20 or 21 non-homeomorphic such manifolds. For every manifold case not marked with $*$ in the table, the universal regular tessellation is known to be a link complement. Thus there are 19 to 21 regular tessellation link complements.
\end{theorem}

\begin{table}[h]
\caption{Size of the universal regular tessellation $\MunivPrinCong{\pqr}{z}$ in number of Platonic solids and cusps (also see Figure~\ref{fig:OverviewArithCases}).}\label{table:finiteVolUnivReg}
\begin{tabular}{r|c|c||c|c||c|c||r|c|c}
\multicolumn{3}{c||}{$\{3,3,6\}$} & \multicolumn{2}{c||}{$\{4,3,6\}$} & \multicolumn{2}{c||}{$\{5,3,6\}$} & \multicolumn{3}{c}{$\{3,4,4\}$}\\
\multicolumn{3}{c||}{tetrahedron} & \multicolumn{2}{c||}{cube} & \multicolumn{2}{c||}{dodecahedron} & \multicolumn{3}{c}{octahedron}\\
$z$ & solids & cusps & solids & cusps & solids & cusps & $z$ & solids & cusps\\ \hline \hline
1 & \multicolumn{2}{|c||}{orbifold} &  \multicolumn{2}{|c||}{orbifold} &  \multicolumn{2}{|c||}{orbifold} & 1 &  \multicolumn{2}{|c}{orbifold}\\ \hline
$1+\zeta$ & \multicolumn{2}{|c||}{orbifold} & 6 & 8 & \multicolumn{2}{|c||}{orbifold} & $1+i$ & \multicolumn{2}{|c}{orbifold}\\ \hline
$2$ & 10 & 5 & 16 & 16 & 240 & 600 & $2$ & 4 & 6\\ \hline
$2+\zeta$ & 28 & 8 & 84 & 48 & \multicolumn{2}{|c||}{?} & $2+i$ & 5 & 6 \\ \hline
$2+2\zeta$ & 120 & 20 & \multicolumn{2}{|c||}{} & \multicolumn{2}{|c||}{} & $2+2i$ & 16 & 12\\ \hline
$3$ & 54 & 12 & \multicolumn{2}{|c||}{} & \multicolumn{2}{|c||}{} & $3$ & 30 & 20 \\ \hline
$3+\zeta$ & 182 & 28 & \multicolumn{2}{|c||}{} & \multicolumn{2}{|c||}{} & $3+i$ & 30 & 18 \\ \hline
$3+2\zeta$ & 570 & 60 & \multicolumn{2}{|c||}{} & \multicolumn{2}{|c||}{} & $3+2i$ & 91 & 42 \\ \hline
$4$ & 640 & 80 & \multicolumn{2}{|c||}{} & \multicolumn{2}{|c||}{} & $4+i$ & 204 & 72 \\ \hline
$4+\zeta$ & 672 & 64 & \multicolumn{2}{|c||}{} & \multicolumn{2}{|c||}{} & $4+2i$ & 122880 & $36864^*$ \\ \hline
\end{tabular}
\end{table}

Note that there are two unsettled cases $\MunivPrinCong{\{5,3,6\}}{2+\zeta}$ and $\MunivPrinCong{\{3,4,4\}}{4+2i}$. In the first case, it is unclear whether the universal regular tessellation $\MunivPrinCong{\{5,3,6\}}{2+\zeta}$ is finite-volume. Computer experiments strongly indicate that it is indeed infinite volume and thus not a link complement; see also the discussion in Section~\ref{section:dicussion}. In the second case, we know that the universal regular tessellation $\MunivPrinCong{\{3,4,4\}}{4+2i}$ is finite-volume, but could neither prove nor disprove that it is a link complement.

This leaves the question unsettled whether every finite volume manifold universal regular tessellation is a link complement. Baker and Reid posed a more general formulation of this question in \cite{bakerReid:prinCong}. We will discuss this further in the discussion section.

For some cases, a link with complement $\MunivPrinCong{\pqr}{z}$ could be found easily or has been explicitly constructed: $\MunivPrinCong{\{3,3,6\}}{2}$ is shown in Figure~\ref{fig:5chainLink}, the complement of the Thurston congruence link \cite{thurston:HowToSee} is $\MunivPrinCong{\{3,3,6\}}{2+\zeta}$, links for $\MunivPrinCong{\{3,3,6\}}{2+2\zeta}$ and $\MunivPrinCong{\{3,3,6\}}{3}$ have been constructed in the author's PhD thesis \cite{goerner:thesis}. Finally, the complement of the minimally twisted 6-component chain link turns out to be $\MunivPrinCong{\{3,4,4\}}{2}$.

The rest of the paper is organized as follows: We first develop the theory of regular tessellations in Section~\ref{sec:regTessTorus} and \ref{sec:regTessSolids}. An important invariant is the cusp modulus describing the regular tessellation structure when restricting to a cusp neighborhood. The largest tessellation and only potential link complement among all tessellations of given cusp modulus is the universal regular tessellation. 

Section~\ref{sec:FinUnivRegTess} classifies all regular tessellations for small cusp modulus and all principal congruence link complements for discriminant $D=-3$ and $D=-4$.

Section~\ref{sec:ConstructUnivRegTess} introduces the central algorithm in this paper: the construction of the universal regular tessellation. This has been implemented in python (see \url{http://www.unhyperbolic.org/regTess/}) and Section~\ref{sec:ImplDetails} describes the details of the implementation as well as examples of how to use the software to obtain the results in Table~\ref{table:finiteVolUnivReg} and Section~\ref{sec:FinUnivRegTess}.

The rest of the paper is devoted to proving the main theorem which means proving that the algorithm to construct the universal regular tessellation never terminates for a case not listed in Table~\ref{table:finiteVolUnivReg}. Section~\ref{sec:proofInfLarge} does this for large cusp modulus. Section~\ref{sec:cuspidalCovers} introduces cuspidal homology to construct cuspidal covers. This is used in Section~\ref{sec:infProofSpecial} to prove infinite universal regular tessellations in the remaining cases.

\section{Regular Tessellations of the Torus}\label{sec:regTessTorus}

Let $\calO_D$ be the ring of integers of the imaginary quadratic number field $\Q(\sqrt{D})$ of discriminant $D < 0$. Here, we will focus on the Eisenstein integers $\calO_{-3}=\Z[\zeta]$ with $\zeta=\frac{1+\sqrt{-3}}{2}$ and the Gaussian integers $\calO_{-4}=\Z[i]$. Let $u$ denote the generator $\zeta$, respectively, $i$ of each of these two rings which are principal ideal domains so every ideal is of the form $\langle z\rangle$ for some $z\in \Z[u]$ which is determined by the ideal up to a unit $u^k$.

Draw a line segment between each pair of points in $\Z[u]\subset\C$ that have unit distance. The result is a regular tessellation of type $\{3,6\}$ for $D=-3$, respectively, $\{4,4\}$ for $D=-4$.

\begin{definition}\label{def:regTessTorus}
Given $z\in \Z[u]\setminus 0$, let $T_z$ be the triangulation of the torus obtained as quotient of the above regular tessellation by the action of the elements in the ideal $\langle z \rangle$ by translations. The dual tessellation of the torus is denoted by $T^*_z$.
\end{definition}

\begin{figure}[h]
\begin{center}
\includegraphics[scale=0.3]{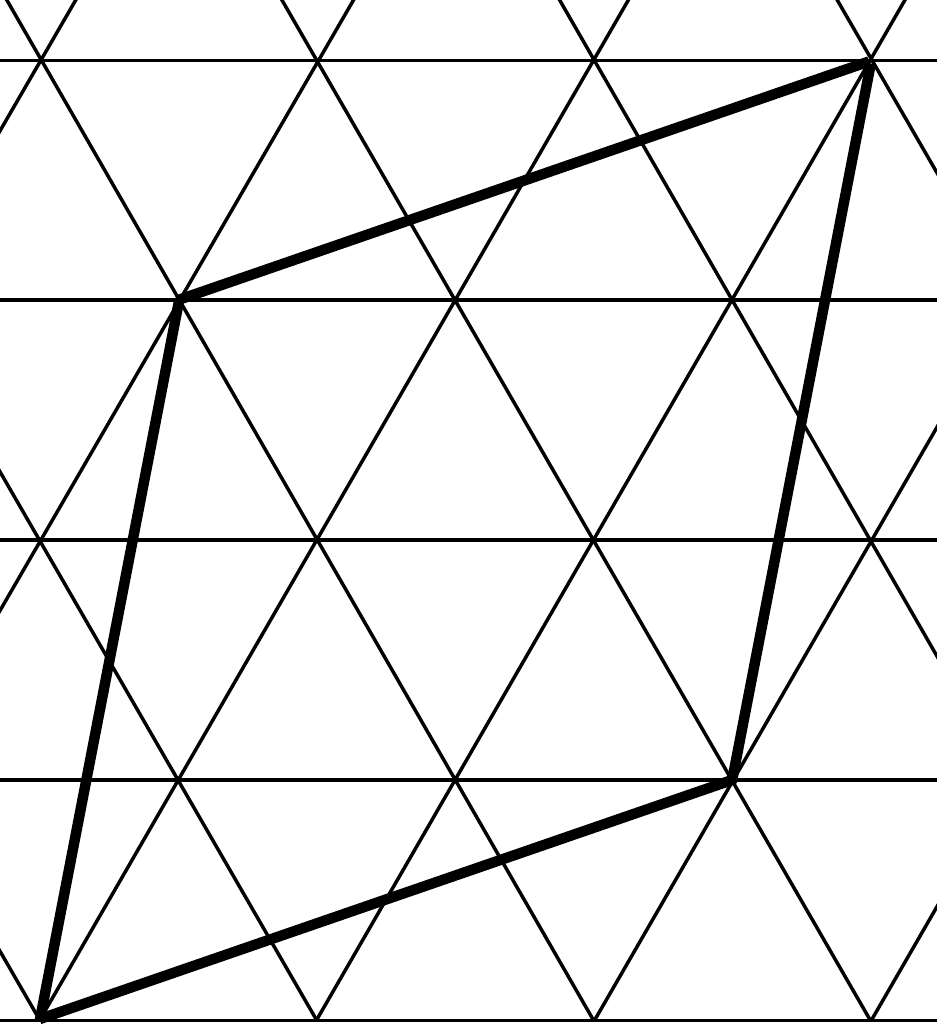}
\end{center}
\caption{Fundamental domain for the chiral regular tessellation $T_{2+\zeta}$ for $D=-3$.}
\label{fig:Torus2pluszeta}
\end{figure}

Recall that every regular tessellation of the torus is of the form $T_z$ or $T^*_z$.  An example is given in Figure~\ref{fig:Torus2pluszeta}. Given an oriented regular tessellation of a torus, the $z$ classifying the tessellation is determined up to multiplication by a unit $u^k$. Note that $T_z$ is the mirror image of $T_{\bar z}$. Hence, given an unoriented regular tessellation of the torus, $z$ is only defined up to multiplication by a unit $u^k$ and complex conjugation. Furthermore, $T_z$ is chiral if and only if $\langle z\rangle \not= \langle \bar z\rangle$.

\section{Cusp Modulus} \label{sec:regTessSolids}

Recall that we normalized the regular tessellations $\{3,3,6\}, \{3,4,4\}, \{4,3,6\}, \{5,3,6\}$ such that there is a face with three consecutive vertices at the points $\infty, 0$, and $1$, see Figure~\ref{fig:hypOct}. Take a horosphere $H$ about $\infty$ that is a high enough plane parallel to $\C$. A solid of an ideal regular hyperbolic tessellations $\pqr$ intersects $H$ in a regular $q$-gon and thus the regular tessellation of $\H^3$ induces a regular tessellation of type $\{q, r\}$ on $H$ with vertices at $\Z[u]$ where $u=e^{2i\pi /r}$. The orientation-preserving isometries of the tessellation on $H$ are given by the upper triangular matrices with coefficients in $\Z[u]$ with upper unit-triangular matrices corresponding to translations. In fact, $\tetGroup{\pqr}$ coincides with the natural $\Z/2$-extension $\PGL(2,\calO_D)$ of a Bianchi group $\PSL(2,\calO_D)$ in two cases: $\tetGroup{\{3,3,6\}}$ is given by $\PGL(2,\calO_{-3})$ and $\tetGroup{\{3,4,4\}}$ by $\PGL(2,\calO_{-4})$. One of the other two cases, $\tetGroup{\{4,3,6\}}$, is also arithmetic because it is commensurable with $\PGL(2,\calO_{-3})$ but does neither cover it nor is covered by it. However, $\tetGroup{\{5,3,6\}}$ is not arithmetic.  As group, each $\tetGroup{\pqr}$ is the orientation-preserving index-2 subgroup of a Coxeter reflection group.

\begin{figure}[h]
\begin{center}
\scalebox{0.7}{\input{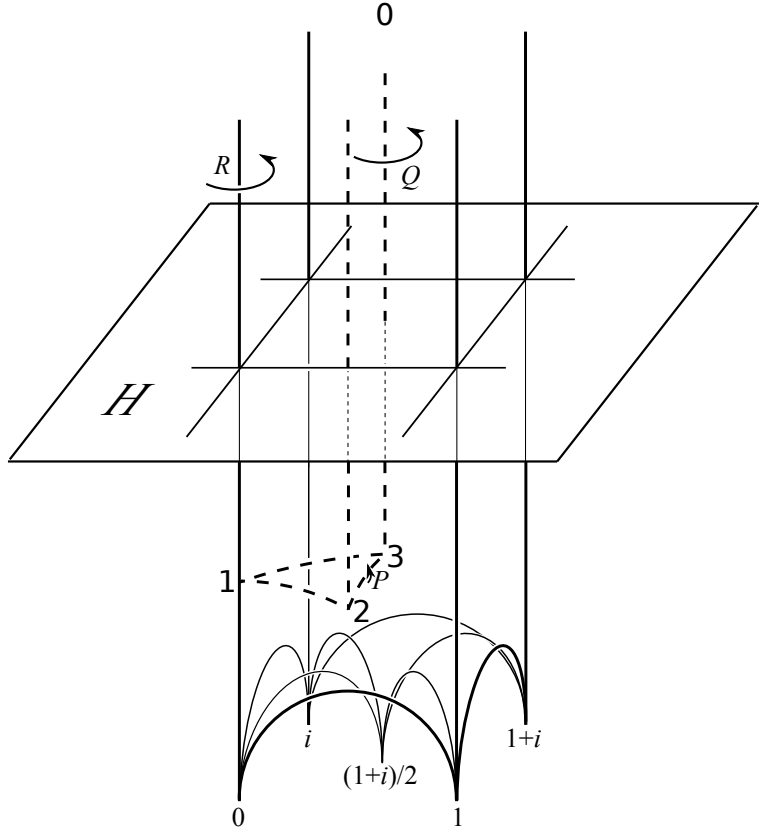}}
\end{center}
\caption{An octahedron in the regular tessellation $\{3,4,4\}$. The dashed lines form one of the simplices obtained from a barycentric subdivision. The octahedron intersects the horosphere $H$ in a square so there is an induced regular tessellation $\{4,4\}$ on $H$.   \label{fig:hypOct}}
\end{figure}

Now consider a regular tessellation $\Mmfd$ of a finite-volume oriented cusped hyperbolic manifold and of type $\pqr$. Similarly, there is an induced oriented tessellation $\{q,r\}$ of the boundary of a cusp neighborhood for each cusp, this time on a torus. The tessellation of the 3-manifold being regular also implies that the induced tessellation is regular and the same for each cusp. Thus, the induced oriented regular tessellation is given by some $T_z$.
Thus, we obtain an invariant $z$ of $\Mmfd$ defined up to multiplication by a unit $u^k$ which we call the {\bf cusp modulus}. Note that for the non-arithmetic case $\{5,3,6\}$, the cusp modulus is still an element in $\calO_{-3}$ but that $\calO_{-3}$ is not the invariant trace field. 

\begin{remark}
The regular tessellation structure is also equal to the canonical cell decomposition of the underlying hyperbolic manifold \cite{EpsteinPenner:CanonicalDecomp}. In particular, a hyperbolic 3-manifold admits at most one regular tessellation structure and the cusp modulus is an invariant of the underlying hyperbolic 3-manifold admitting a regular tessellation. Note that this is in general false for non-regular tessellations by Platonic solids. For example, the manifold underlying the regular tessellation $\MunivPrinCong{\{3,3,6\}}{2}$ of ten tetrahedra also admits several different tessellations by two regular cubes but none of the cubical tessellations is regular and thus equal to the canonical cell decomposition ($\MunivPrinCong{\{3,3,6\}}{2}$ is a double cover of the orbifold in Figure~\ref{fig:5chainLink}. This orbifold was named $Q_4$ in \cite{neumannReid:SmallVolOrbs} and shown to decompose into a cube).
\end{remark}

We can also characterize tessellations algebraically. A manifold $\Mmfd=\H^3/\mfd$ has a tessellation induced from the tessellation $\pqr$ of $\H^3$ if $\mfd$ is a torsion-free subgroup of $\tetGroup{\pqr}$. The induced tessellation on $\Mmfd$ is a regular tessellation if furthermore $\mfd$ is normal in $\tetGroup{\pqr}$ (this is how it was defined earlier in Definiton~\ref{def:algDefRegTess}). To see this, recall that a tessellation is regular if its orientation-preserving symmetry group acts transitively on all flags. $\tetGroup{\pqr}$ acts freely and transitively on all flags of the tessellation of $\H^3$. Thus a manifold quotient $\H^3/\mfd$ is a regular tessellation if every symmetry of $\tetGroup{\pqr}$ descends to a  symmetry of $\H^3/\mfd$ which is equivalent to $\mfd$ being normal in $\tetGroup{\pqr}$.

Let $$p_z=\left(\begin{array}{cc}1 & z\\ 0 & 1\end{array}\right),\qquad P=\left\{p_w : w \in \C \right\}\quad\mbox{and}\quad P_z=\left\{p_w : w\in\langle z\rangle\right\}.$$

\begin{definition}
Let $\Mmfd$ be a cusped orientable 3-manifold such that $\Mmfd=\H^3/\mfd$ where $\mfd$ is a torsion-free normal subgroup of $\tetGroup{\pqr}$. Then $z$ is the cusp modulus of $\Mmfd$ if $\mfd\cap P=P_z$.
\end{definition}

Note that in this definition, the cusp modulus is defined up to a multiplication by a unit $u^k$ again. To see this, look at the group $\mfd\cap P$. If $p_w$ and $p_{w'}\in \mfd\cap P$, then $p_{w+w'}=p_w p_{w'}\in\mfd\cap P$ and $p_{uw}=g_R p_w g_R^{-1} \in\mfd\cap P$ where $g_R=\left(\begin{array}{cc}u & 0\\ 0 & 1\end{array}\right)\in\tetGroup{\pqr}$ because $M$ is normal. Thus, the off-diagonal entries of the matrices in $\mfd\cap P$ form an ideal and the ideal is generated by an element $z$ determined up to $u^k$.

This definition of cusp modulus also matches the earlier geometric definition of cusp modulus: the group $\mfd\cap P$ acts on the regular tessellation $\{q,r\}$ on the horosphere $H$ by translations and the resulting quotient is the tessellated torus $T_z$. 

Recall the universal regular tessellation $\MunivPrinCong{\pqr}{z}$ from Definition~\ref{def:univRegTess}. If it is a manifold, it is a regular tessellation by definition and it always has the following universal property:



\begin{lemma}
Let $\Mmfd$ be a (not necessarily finite volume) regular tessellation $\pqr$ with cusp modulus $z$. Then there is a covering map $\MunivPrinCong{\pqr}{z}\to\Mmfd$.
\end{lemma}

\begin{proof}
The group $\mfd$ contains $p_z$ and is normal in $\tetGroup{\pqr}$. Thus it must contain $\univPrinCong{\pqr}{z}$ by definition.
\end{proof}

\begin{lemma} \label{lemma:linkImpliesUniv}
If a regular tessellation $\Mmfd$ is a link complement, then it is also a universal regular tessellation $\MunivPrinCong{\pqr}{z}$.
\end{lemma}

\begin{proof}
A regular tessellation $\Mmfd$ is universal if and only if the fundamental group $\mfd$ is generated by parabolic elements only. Note that the Wirtinger representation of the fundamental group of hyperbolic link complement implies that it also has this property.
\end{proof}

\section{Classification of Regular Tessellations with Small Cusp Modulus} \label{sec:FinUnivRegTess}


Once we have constructed a finite-volume universal regular tessellation for a cusp modulus $z$, we can construct all regular tessellation of cusp modulus $z$ using the techniques described later in Section~\ref{sec:allRegTessThorughGap}. This is done in Table~\ref{table:FinCat} by listing the following categories to capture the relationships between these regular tessellations:

\begin{definition} \label{def:catRegTessCuspMod}
Let $\mathcal{C}^{\pqr}_z$ denote the category of all pointed manifold regular tessellations of type $\pqr$ and cusp modulus $z$. Recall that each regular tessellation is a cover of the orbifold $\H^3/\tetGroup{\pqr}$ and pick a generic point $p_0$ in the orbifold (that is not on the singular locus or preserved under any symmetries of the orbifold). An object $(\Mmfd, p)\in\mathcal{C}^{\pqr}_z$ is a regular tessellation of type $\pqr$ and cusp modulus $z$ such that $p$ is a lift of $p_0$. A morphism is a covering map respecting the base point $p$.
\end{definition}

\begin{table}[h]
\caption{Categories $\mathcal{C}_z^{\pqr}$ up to category equivalence for small cusp modulus.
}\label{table:FinCat}
\begin{center}
\scalebox{0.9}{
$\mathcal{C}^{\pqr}_z \cong \big\{\MunivPrinCong{\pqr}{z}\big\}\quad\mbox{for}\quad \mathcal{C}^{\{3,3,6\}}_2, \mathcal{C}^{\{3,3,6\}}_{2+\zeta}, \mathcal{C}^{\{3,3,6\}}_3, \mathcal{C}^{\{3,3,6\}}_{3+2\zeta}, \mathcal{C}^{\{4,3,6\}}_{1+\zeta}, \mathcal{C}^{\{3,4,4\}}_2, \mathcal{C}^{\{3,4,4\}}_{2+i}, \mathcal{C}^{\{3,4,4\}}_{3+2i}$
}

\medskip

\scalebox{0.9}{$\mathcal{C}^{\{3,3,6\}}_{2+2\zeta} \cong\Big\{\MexceptionalObject{-3}{2+2\zeta} \xleftarrow{(\Z/2)^2} \MpglPrinCong{-3}{2+2\zeta}\Big\}$ \quad
$\mathcal{C}^{\{3,3,6\}}_{3+\zeta} \cong\Big\{\MpglPrinCong{-3}{3+\zeta} \xleftarrow{\Z/2} \MpslPrinCong{-3}{3+\zeta}\Big\}$}

\scalebox{0.9}{
$\xymatrixcolsep{0.5pc}\mathcal{C}^{\{3,3,6\}}_{4}\cong \Big\{\xymatrix@C+5mm{\MpslPrinCong{-3}{4} & \ar[l]^{\Z/2} \MpslPrinCongBig{-3}{4} & \ar[l]^{\Z/2} \ar@/_0.8pc/[ll]_{\Z/4} \MunivPrinCong{\{3,3,6\}}{4}}\Big\}$ \qquad
$\xymatrixcolsep{0.7pc}\mathcal{C}^{\{3,3,6\}}_{4+\zeta} \cong \Big\{\xymatrix@C+5mm{\MexceptionalObject{-3}{4+\zeta} & \ar[l]^{(\Z/2)^2}  \MpslPrinCong{-3}{4+\zeta} & \ar[l]^{\Z/2} \ar@/_0.8pc/[ll]_{Q_8} \MpslPrinCongBig{-3}{4+\zeta}}\Big\}$
}

\scalebox{0.9}{
$ \mathcal{C}^{\{4,3,6\}}_2\cong \Bigg\{\begin{minipage}{6.5cm} ${\xymatrixrowsep{-0.4pc} \xymatrixcolsep{1.1pc}\xymatrix@C+5mm{\mathcal{Z}_0\\ & \ar@/^0.1pc/[dl]_{(\Z/2)^2} \ar@/_0.1pc/[ul]^{}  \mathcal{Z}_1 & \ar[l]|-{\Z/2}  \ar@/_0.3pc/[llu]_{Q_8} \ar@/^0.3pc/[lld]^{Q_8} \ar[r]^{~\quad~\Z/2} \MunivPrinCong{\{4,3,6\}}{2} & \mathcal{Z}_2 \\ \bar{\mathcal{Z}}_0 }}$ \end{minipage}\Bigg\}$ \quad
$\xymatrixcolsep{0.2pc}\mathcal{C}^{\{4,3,6\}}_{2+\zeta}\cong \Big\{\xymatrix@C+5mm{\mathcal{Z}_0 & \ar[l]^{\Z/2}  \mathcal{Z}_1 & \ar[l]^{\Z/3} \ar@/_0.8pc/[ll]_{S_3} \MunivPrinCong{\{4,3,6\}}{2+\zeta}}\Big\}$
}

\scalebox{0.89}{
$ \mathcal{C}^{\{5,3,6\}}_2\cong \Bigg\{\begin{minipage}{7.3cm} ${\xymatrixrowsep{-0.4pc} \xymatrixcolsep{1.7pc}\xymatrix@C+5mm{\mathcal{Z}_0\\ & \ar@/_0.2pc/[dl]_{A_5} \ar@/^0.2pc/[ul]^{}  \mathcal{Z}_1 & \ar[l]|-{\Z/2}  \ar@/_0.3pc/[llu]_{\SL(2,5)} \ar@/^0.3pc/[lld]^{\SL(2,5)} \ar[r]^{~\quad~\Z/2} \MunivPrinCong{\{5,3,6\}}{2} & \mathcal{Z}_2 \\ \bar{\mathcal{Z}}_0 }}$ \end{minipage}\Bigg\}$}\\
\scalebox{0.9}{
$\mathcal{C}^{\{3,4,4\}}_{2+2i}  \cong    \big\{ \MexceptionalObject{-4}{2+2i} \xleftarrow{\Z/2}  \MpglPrinCong{-4}{2+2i} \big\}$ \qquad
$\mathcal{C}^{\{3,4,4\}}_{3}  \cong  \big\{ \MpglPrinCong{-4}{3} \xleftarrow{\Z/2} \MpslPrinCong{-4}{3}\big\}$}

\scalebox{0.9}{
$\mathcal{C}^{\{3,4,4\}}_{3+i}  \cong \big\{ \MexceptionalObject{-4}{3+i} \xleftarrow{\Z/3} \MpglPrinCong{-4}{3+i}\big\}$ \qquad
$\mathcal{C}^{\{3,4,4\}}_{4+i}  \cong  \big\{\MpglPrinCong{-4}{4+i} \xleftarrow{\Z/2} \MpslPrinCong{-4}{4+i}\big\}$
}
\scalebox{0.9}{
$\mathcal{C}^{\{3,4,4\}}_{4+2i} \cong \big\{
\xymatrix@C-1mm{\cdots ~~~ \MpslPrinCong{-4}{4+2i} & \ar[l]_{~\quad~\Z/2} \MpslPrinCongBig{-4}{4+2i} & \cdots & \ar@/_1.3pc/[ll]_{H}^{|H|=512}  \MunivPrinCong{-4}{4+2i} }\big\}$
}
\end{center}
\end{table}

For the categories $\mathcal{C}^{\{3,3,6\}}_z$ and $\mathcal{C}^{\{3,4,4\}}_z$, we can describe three objects as quotients by an arithmetically defined group if such a quotient happens to be a manifold. These objects are congruence manifolds with discriminant $D=-3$ (for $\{3,3,6\}$) or $D=-4$ (for $\{3,4,4\}$) and among them are the principal congruence manifolds $\MpslPrinCongBig{D}{z}$: \label{sec:PrinCongDef}
\begin{eqnarray*}
\MpslPrinCongBig{D}{z}=\H^3/\pslPrinCong{D}{z}&\quad\mbox{where}\quad&\pslPrinCongBig{D}{z}=\ker\Big(\PSL\big(2,\calO_D\big)\to \SL\big(2,\calO_D/\langle z\rangle\big)/{\pm 1}\Big)\\
\MpslPrinCong{D}{z}=\H^3/\pslPrinCong{D}{z}&\quad\mbox{where}\quad&\pslPrinCong{D}{z}=\ker\Big(\PSL\big(2,\calO_D\big)\to \PSL\big(2,\calO_D/\langle z\rangle\big)\Big),\\
\MpglPrinCong{D}{z}=\H^3/\pglPrinCong{D}{z}&\quad\mbox{where}\quad&\pglPrinCong{D}{z}=\ker\Big(\PGL\big(2,\calO_D\big)\to \PGL\big(2,\calO_D/\langle z\rangle\big)\Big).
\end{eqnarray*}
Here, $\PGL(2,R)=\mathrm{GL}(2,R)/R^*$ and $\PSL(2,R)=\SL(2,R)/\{e\in R|e^2=1\}$. For $R=\calO_D$, $\pm 1$ are the only two elements $e$ with $e^2=1$, so the map to $\SL\big(2,\calO_D/\langle z\rangle\big)/{\pm 1}$ is well-defined. 

\begin{lemma}\label{lemma:prinCongIsRegTess}
If a space $\MpslPrinCongBig{D}{z}, \MpslPrinCong{D}{z},$ respectively, $\MpglPrinCong{D}{z}$ is a manifold, then it is a regular tessellation. Thus, there are covering maps (some might be isomorphisms)
$$\MunivPrinCong{\pqr}{z}\to \MpslPrinCongBig{D}{z}\to\MpslPrinCong{D}{z}\to\MpglPrinCong{D}{z}.$$
\end{lemma}
\begin{proof}
We need to show that all the groups $\pslPrinCongBig{D}{z}, \pslPrinCong{D}{z}, \pglPrinCong{D}{z}$ are normal in the respective $\Gamma^{\{3,q,r\}}$. This is obvious for the kernel $\pglPrinCong{D}{z}$ since $\PGL(2,\calO_{-3})=\Gamma^{\{3,3,6\}}$ and $\PGL(2,\calO_{-4})=\Gamma^{\{3,4,4\}}.$ It is also obvious that $\pslPrinCong{D}{z}$ and $\pslPrinCongBig{D}{z}$ are normal in $\PSL(2,\calO_d)$ which is an index-2 subgroup of $\PGL(2,\calO_d)$ with an element in the complement being $g=\left(\begin{array}{cc}u & 0\\ 0 & 1\end{array}\right)$. Hence, it is enough to show that the conjugate $M^g$ of a matrix $M$ in $\pslPrinCong{D}{z}$ or $\pslPrinCongBig{D}{z}$ is also in $\pslPrinCong{D}{z}$, respectively, $\pslPrinCongBig{D}{z}$. This follows from $M^g$ still being in $\PSL(2,\calO_D)$ and its image in $\PSL(2,\calO_D/\langle z\rangle)$, respectively, $\PGL(2,\calO_D/\langle z\rangle)$ still being congruent to the identity matrix.
\end{proof}

If there are other objects in $\mathcal{C}^{\pqr}_z$ not listed in the above lemma, we denote them by $\mathcal{Z}_{(j)}$ and $\bar{\mathcal{Z}}_{(j)}$. The latter one is used for the mirror image of $\mathcal{Z}_{(j)}$ when $\mathcal{Z}_{(j)}$ is chiral. $Q_8$ denotes the quaternion group of order 8.

We can now classify all principal congruence link complements as a corollary of the main theorem. It is left to run the algorithm described later in Section~\ref{sec:constructPrinCong} to check whether the arithmetically defined space is an orbifold or manifold. We obtain an orbifold for the arithmetically defined spaces in exactly the following cases: $\MpslPrinCongBig{D}{z}, \MpslPrinCong{D}{z}, \MpglPrinCong{D}{z}$ with $|z|<2$ and $\MpglPrinCong{-4}{2}, \MpslPrinCong{-4}{2}, \MpglPrinCong{-3}{2}$.

\begin{figure}
\begin{center}
\scalebox{0.65}{
\input{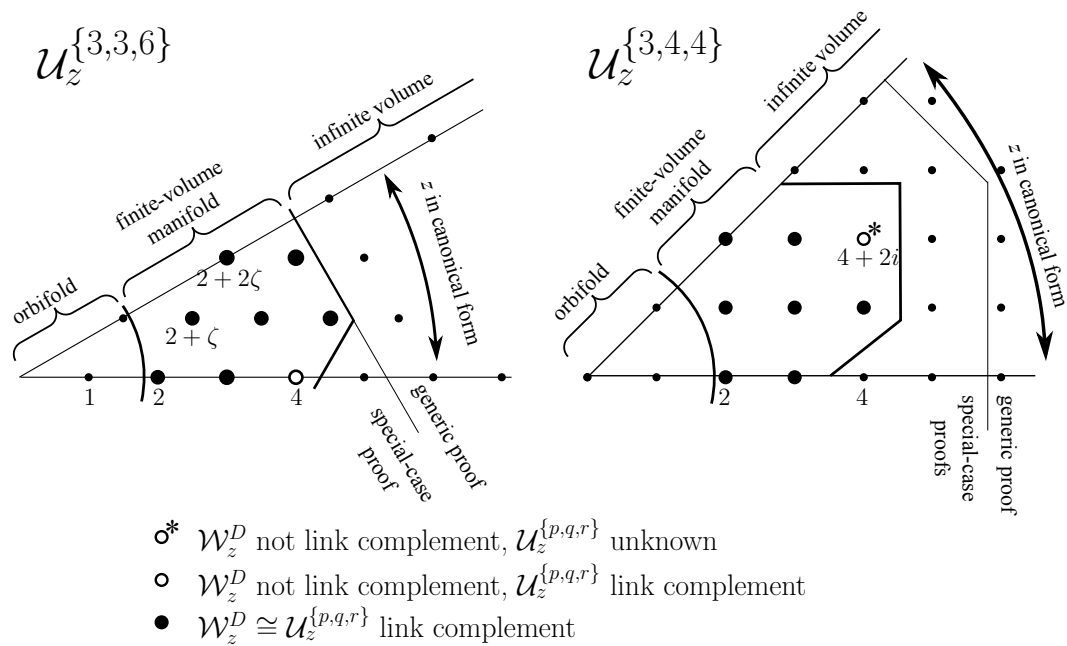}}
\end{center}
\caption{Overview of regular tessellation link complements and principal congruence link complements for $\{3,3,6\}$, $D=-3$ and $\{3,4,4\}$, $D=-4$.\label{fig:OverviewArithCases}}
\end{figure}

\newpage

\begin{cor}\label{cor:classPrinCong}
A principal congruence manifold $\MpslPrinCongBig{D}{z}$ with $D=-3, -4$ and $z$ in canonical form is a link complement if and only if:
\begin{itemize}
\item $D=-3$ and $z\in\{2, 2+\zeta, 2+2\zeta, 3, 3+\zeta, 3+2\zeta, 4+\zeta\}$
\item $D=-4$ and $z\in\{2, 2+ i, 2+2i, 3, 3+i, 3+2i, 4+i\}$
\end{itemize}
\end{cor}

\begin{proof}
By Lemma~\ref{lemma:prinCongIsRegTess} and Theorem~\ref{thm:main}, $\MpslPrinCongBig{D}{z}$ can only be a link complement if it is a finite volume universal regular tessellation. There are 16 potential cases of type $\{3,3,6\}$ and $\{3,4,4\}$. In all these cases, $\MpslPrinCongBig{D}{z}$ is a manifold. But in the case $\MpslPrinCongBig{-3}{4}$ and $\MpslPrinCongBig{-4}{4+2i}$, it is not the universal regular tessellation, leaving $7$ cases for each discriminant. An overview is also given in Figure~\ref{fig:OverviewArithCases}.
\end{proof}

\section{Construction of the Universal Regular Tessellation} \label{sec:ConstructUnivRegTess}

In this section, we give an algorithm to construct $\MunivPrinCong{\pqr}{z}$  which will terminate if and only if $\MunivPrinCong{\pqr}{z}$ is finite-volume. Recall the definition of the torus $T^*_z$ from Definition~\ref{def:regTessTorus}.

\begin{definition}\label{def:Nanotube}
A nanotube is the product $\R^{\geq 0} \times T^*_z$ where $T^*_z$ is tessellated by regular $r$-gons when constructing tessellations of type $\pqr$ and cusp modulus $z$.
\end{definition}

\begin{figure}[h]
\begin{center}
\includegraphics[scale=0.25]{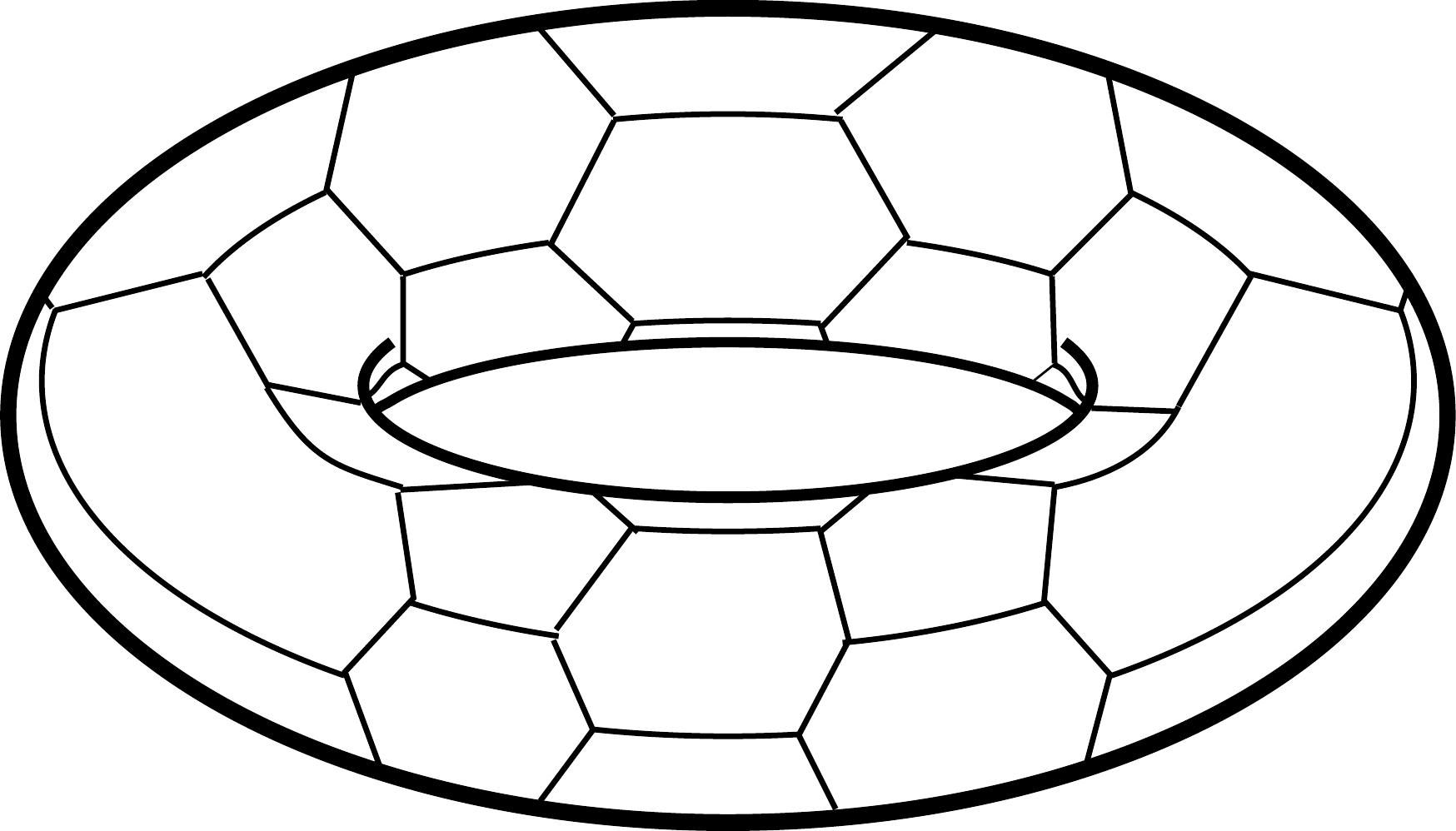}
\end{center}
\caption{A nanotube $\R^{\geq 0} \times T^*_z$ for tessellations of type $\{p,3,6\}$. The core of the torus is removed.\label{fig:nanotube}}
\end{figure}

Figure~\ref{fig:nanotube} shows an example nanotube. Whereas a regular tessellation $\Mmfd$ is built from ideal Platonic solids, its dual $\Mmfd^*$ is built from topological nanotubes $\R^{\geq 0} \times T^*_z$ such that $p$ nanotubes meet at each edge and $z$ is the cusp modulus. The reader is probably already familiar with the decomposition of a regularly tessellated cusped hyperbolic manifold into nanotubes because the nanotubes happen to be the Ford domains of the manifold. This is analogous to the regular tessellation by Platonic solids being identical to the canonical cell decomposition \cite{EpsteinPenner:CanonicalDecomp}. However, we prefer the dual nanotubes as building blocks in this section because they already have the cusp modulus encoded in them, thus making it easier to build tessellation with prescribed cusp modulus. The use of the dual as well as the term ``nanotube'' was first suggested by Ian Agol.

To obtain the universal regular tessellation $\MunivPrinCong{\pqr}{z}$, we can (roughly speaking) just glue enough nanotubes together and enforce that there are always $p$ nanotubes at an edge. A deterministic algorithm for this is described in the following definition.
\begin{definition} \label{def:UnivRegTessAlgo}
Let $\MunivPrinCong{\pqr}{z}(0)$ be a single nanotube. To obtain $\MunivPrinCong{\pqr}{z}(n+1)$ from $\MunivPrinCong{\pqr}{z}(n)$, perform the following steps:
\begin{enumerate}
\item Attach a nanotube to each open face of $\MunivPrinCong{\pqr}{z}(n)$.
\item If there is an edge $e$ of the resulting complex with $p$ nanotubes around $e$ and two open faces adjacent to $e$, we need to glue the two faces so we a get an edge cycle about $e$. If we already had an edge cycle about an edge $e$ of length different from $p$ or there are more than $p$ nanotubes adjacent to $e$, we need to identify every $p$-th nanotube. Repeat until there is no edge left for which such a gluing or identification is necessary.
\end{enumerate}
\end{definition}

\begin{figure}[h]
\begin{center}
\includegraphics[scale=0.35]{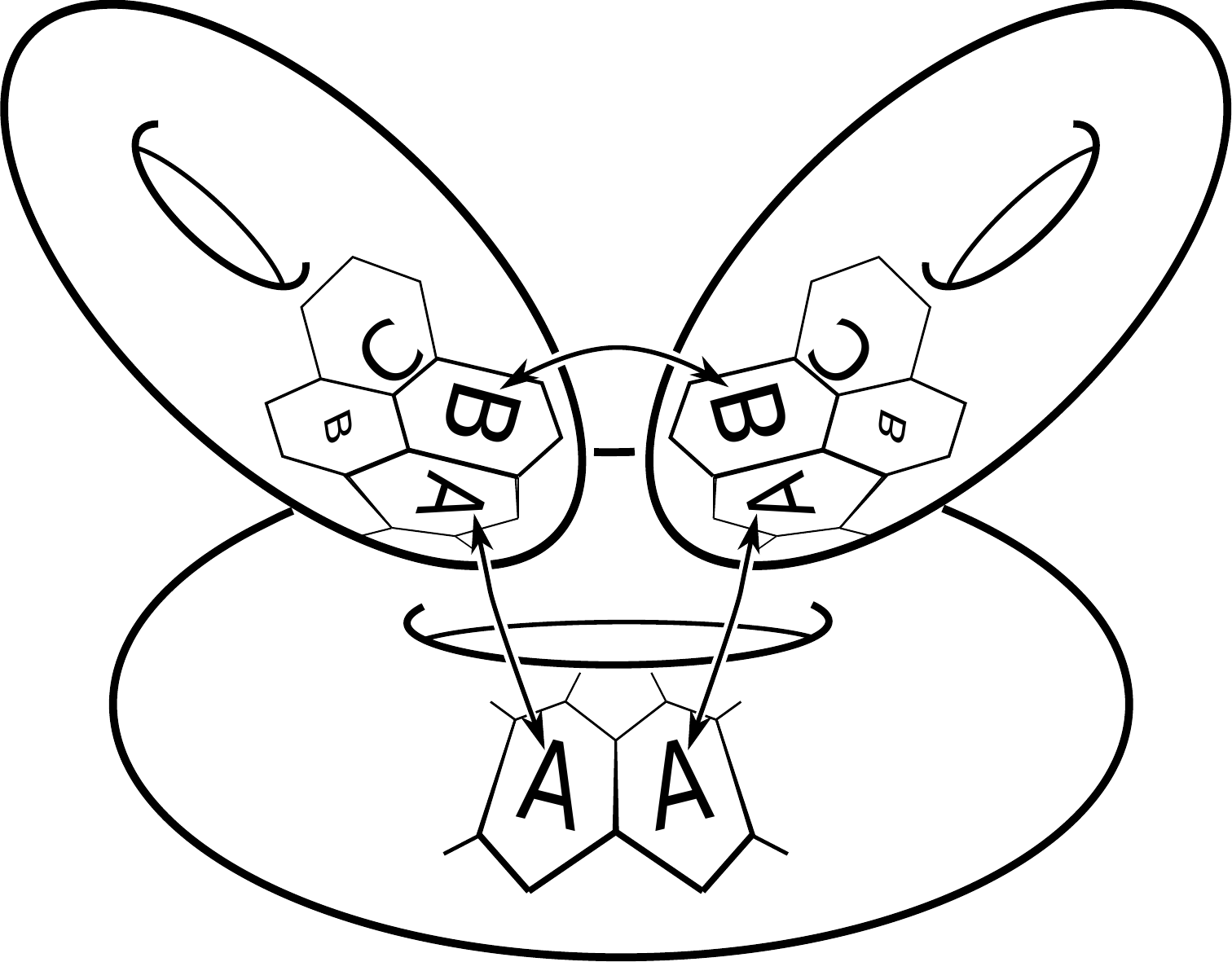}
\end{center}
\caption{Attaching nanotubes to $\MunivPrinCong{\{3,3,6\}}{z}(0)$. Only two of the new nanotubes are shown here, those will be attached along faces $A$. In step 2, the $B$ faces will be glued. Thus, the edges of an $A$ face are glued up to 3-cycles. The edge that the $C$ faces share with a $B$ face will be labeled by 2.\label{fig:nanotubeAttach}}
\end{figure}

\begin{remark}
$\MunivPrinCong{\pqr}{z}(n+1)$ is well-defined as the result is independent of the order in which we process the edges in step 2. To see this, note that the space $\MunivPrinCong{\pqr}{z}(n+1)$ is the quotient of the complex obtained from step 1 under a certain equivalence relationship. Namely, this relationship is the minimal equivalence relationship closed under the operation of rotating a point near any edge $e$ to the next nanotube $p$ times. It does not matter in which order we perform this operation to obtain the minimal relationship closed under this operation.
\end{remark}

The open faces of $\MunivPrinCong{\pqr}{z}(n)$ form a 2-complex, but since we never glue two nanotubes along edges, only along faces, they even form a surface (or potentially, a 2-orbifold, but only for small $z$, see Remark~\ref{remark:orbifold}). This surface is orientable, connected, and tessellated by $r$-gons. We will label an edge $e$ of $\partial \MunivPrinCong{\pqr}{z}(n)$ with a numeral 1, \dots, $p-1$ indicating the number of nanotubes in $\MunivPrinCong{\pqr}{z}(n)$ that are adjacent to $e$.

An example for $\{3,3,6\}$ is shown in Figure~\ref{fig:nanotubeAttach}. In later iterations, there are edges labeled by $2$. Choose one, say $e$. We again attach two nanotubes to the two open faces adjacent to $e$, so we have four nanotubes around $e$ now. Thus, we need to identify the two newly attached nanotubes with each other. As before, the edge $e$ will become a 3-cycle and disappear from $\MunivPrinCong{\pqr}{z}(n+1)$. Effectively, we added only one nanotube glued to the two open faces of $\MunivPrinCong{\pqr}{z}(n)$ adjacent to $e$ to close up the edge to a 3-cycle.

\begin{remark}
If the map $\MunivPrinCong{\pqr}{z}(n)\to\MunivPrinCong{\pqr}{z}(n+1)$ is an embedding, the edge-labeled tessellation of $\partial\MunivPrinCong{\pqr}{z}(n+1)$ can be determined purely from $\partial\MunivPrinCong{\pqr}{z}(n)$, so it is enough to look at the surface for studying the evolution of the algorithm.
\end{remark}

\begin{lemma}
\label{lemma:ConstructionTerminates}
Let $\Mmfd$ be a regular tessellation of a finite-volume hyperbolic 3-manifold of type $\pqr$ with cusp modulus $z$. There is a map $f_n:\MunivPrinCong{\pqr}{z}(n)\to \Mmfd$ which is unique once we identify $\univPrinCong{\pqr}{z}(0)$ with a Ford domain of $\Mmfd$. For $n$ large enough, $f$ will be surjective. If $\Mmfd$ is furthermore a universal regular tessellation $\MunivPrinCong{\pqr}{z}$, then $\MunivPrinCong{\pqr}{z}(n)\cong \MunivPrinCong{\pqr}{z}$ for a large enough $n$. In other words, if $\MunivPrinCong{\pqr}{z}$ is finite volume, there is an $n$ such that $\partial \MunivPrinCong{\pqr}{z}(n)$ is empty.
\end{lemma}

\begin{proof}
The existence of $f$ is trivial for $\MunivPrinCong{\pqr}{z}(0)$ and follows inductively for $\MunivPrinCong{\pqr}{z}(n)$ from the construction which enforces only the cusp modulus and edge cycle as relations. 
We say that two cusps of $\Mmfd$ are neighbors if they span an edge in the regular tessellation of $\Mmfd$. Note that the $\mathrm{Im}(f_0)$ covers one cusp of $\Mmfd$, $\mathrm{Im}(f_1)$ covers that cusp and its neighbors, $\mathrm{Im}(f_2)$ covers that cusp, its neighbors, and its neighbors' neighbors and so on. Hence, because $\Mmfd$ is connected and has only finitely many cusps,  $f_n$ will eventually be surjective. Now consider the case that $\Mmfd$ is a finite volume universal regular tessellation. Since the only two relations used in the construction of $\MunivPrinCong{\pqr}{z}(n)$ are the cusp modulus and the edge cycle, the map $f_n$ will be an isomorphism for $n$ large enough. Since $\MunivPrinCong{\pqr}{z}(n)\cong \MunivPrinCong{\pqr}{z}$, there will be no open faces in $\MunivPrinCong{\pqr}{z}(n)$ and $\partial \MunivPrinCong{\pqr}{z}(n)$ will be empty.
\end{proof}

\begin{remark}\label{remark:orbifold}
The algorithm can produce an orbifold for small $z$, for example, $\MunivPrinCong{\{3,3,6\}}{1}$. 
$\partial\MunivPrinCong{\{3,3,6\}}{1}(0)$ consists of a single hexagon. Step 1 doubles the space along the hexagon, thus all edges are closed up and there are two nanotubes about an edge. Step 2 now identifies every third nanotube about an edge and since $\mathrm{gcd}(2,3)=1$, there will be only one nanotube adjacent to $e$ after identification. Thus, we have introduced singular locus of order 3. There is always a map $\MunivPrinCong{\pqr}{z}(n)\to\MunivPrinCong{\pqr}{z}(n+1)$, but in this case, the map will not be an embedding for $n=0$. The cases in which $\MunivPrinCong{\pqr}{z}$ is an orbifold are shown in Table~\ref{table:finiteVolUnivReg}.
\end{remark}

\begin{remark}
If the map $\MunivPrinCong{\pqr}{z}(n)\to\MunivPrinCong{\pqr}{z}(n+1)$ is an embedding for all $n$, the construction also gives an algorithm to solve the word problem for the group $\univPrinCong{\pqr}{z}$.
\end{remark}

\section{Computer Implementation} \label{sec:ImplDetails}

The algorithm described in the previous section has been implemented. The source code and all other files necessary for the reader to easily certify the correctness of the results in this paper are available at \url{http://www.unhyperbolic.org/regTess/}.
We encourage the reader to read and experiment with the well-commented code for details. Each subsection starts with an example how to use the software followed by implementation details.

Furthermore, we want to point out that \texttt{regularTessellations.py} contains a new triangulation data structure that allows not just gluing tetrahedra but also identifying them (see Remark~\ref{remark:recursiveIdentification}). This seems to be a useful feature in general, e.g., for constructing quotient spaces but neither SnapPy \cite{SnapPy:SnapPy} nor Regina \cite{burton:regina} implement it.

\subsection{Generating Triangulations with \texttt{regularTessellations.py}}

\subsubsection{The Universal Regular Tessellation $\MunivPrinCong{\pqr}{z}$} \label{sec:pythonForUniv}

Definition~\ref{def:UnivRegTessAlgo} described the algorithm in terms of nanotubes, but we use triangulations here as they are easier to work with and can also be exported into existing 3-manifold software such as Regina or SnapPy. We obtain the triangulation of a nanotube through the barycentric subdivision, i.e., the subdivision on $\R^{\geq 0}\times T_z^*$ induced from the barycentric subdivision of $T_z^*$. After gluing such subdivided nanotubes, the resulting triangulation is also the barycentric subdivision of the regular tessellation (recall that the unsubdivided nanotubes corresponded to the dual Ford domains but barycentric subdivision is invariant under duality). Note that such a triangulation has finite vertices but still can be imported into SnapPy, as SnapPy performs algorithms to remove finite vertices without changing the topology of the manifold upon import. 

The following example shows how this triangulation of the universal regular tessellation, here $\MunivPrinCong{\{3,3,6\}}{2+\zeta}$, can be constructed. As regular tessellation, $\MunivPrinCong{\{3,3,6\}}{2+\zeta}$ consists of 28 regular ideal tetrahedra. Thus its barycentric subdivision consists of 672 simplices, each ideal tetrahedron contributing 24. The last two lines of code convert the data to a Regina triangulation and write it to a file that can be read with SnapPy. Only those last two lines actually depend on Regina being installed, whereas all other methods work in pure Python 2.x:

\begin{verbatim}
>>> from regularTessellations import *
>>> tessConext = TessellationContext(3,3,6,2,1)
>>> tess = tessContext.UniversalRegularTessellation()
>>> len(tess), len(tess)/24
672, 28
>>> reginaTrig = TetrahedraToReginaTriangulation(tess)
>>> open('2_plus_1_zeta_tets.trig','w').write(reginaTrig.snapPea())
\end{verbatim}

We now describe in more detail how this is implemented and refer the reader to the python code for details.
We first need to write a nanotube factory that produces the triangulation of a nanotube obtained by barycentric subdivision. An algorithm to create the resulting triangulation is easily implemented and we spare the reader with the details, only mentioning the conventions we use here (also see Figure~\ref{fig:hypOct}):
We label the vertices of a simplex in this subdivision such that vertex 0 is ideal, 1 at the center of a face of the nanotube, 2 at an edge adjacent to the face, and 3 at a vertex adjacent to the edge. Thus, face $i$ opposite to vertex $i$ is always glued to face $i$ of another simplex such that the face pairing permutation is trivial, i.e., such that the vertex $j~(\not = i)$ is glued to vertex $j$, and the only gluing data we need to store per simplex is one reference to another simplex per face. We also label the simplices with an orientation so that two neighboring simplices always have opposite orientations since each face gluing reverses the orientation. Every face 0 of a newly created nanotube is unglued.

Step 1 of Definition~\ref{def:UnivRegTessAlgo} just invokes the nanotube factory to create a new nanotube $F$ for each open $r$-gon $A'$ of $\MunivPrinCong{\pqr}{z}(n)$ and then glues some $r$-gon $A$ of $F$ to $A'$. An $r$-gon is formed by $2r$ simplices, so this involves gluing $2r$ simplices of $\MunivPrinCong{\pqr}{z}(n)$ to $2r$ simplices of $F$ along face 0 and we must be careful to glue them in such a way that a positively oriented simplex is glued to a negatively oriented simplex.

Then we need to apply step 2 to each simplex $T$. Let $T_i$ be the simplex glued to face $i$ of $T$. If the face is unglued, we say that $T_i$ does not exist. Similarly, let $T_{a_1\dots a_j}$ be the simplex glued to face $a_j$ of simplex $T_{a_1\dots a_{j-1}}$. If simplex $T_{1010\dots 10}$ exists, then identify $T_{1010\dots 10}$ with $T$ if they are not already identified. Otherwise, glue $T$ to $T_{1010\dots 1}$ along face 0 if face 0 of $T$ is unglued and $T_{1010\dots 1}$ exists. Here, strings such as ``1010\dots 1'' are supposed to contain ``1'' $p$ times.

We need to iterate step 2 until no identification or gluing happened in an iteration. This is because gluing up one edge might trigger that nearby edges have more adjacent simplices and need to be glued up.

It is left to write a loop that repeats step 1 and 2 until there are no open faces.

\begin{remark}\label{remark:recursiveIdentification}
When we identify two simplices $T$ and $T'$, the identification needs to be pushed through the already existing gluings, e.g., if $T_0$ is glued to $T$ along face 0 and $T'_0$ is glued to $T'$, then $T_0$ and $T'_0$ need to be identified as well if they are different, and this identification then needs to be recursively pushed through as well. If only $T$ is glued to a simplex $T_0$, but  face 0 of $T'$ is unglued, the simplex resulting from identifying $T$ and $T'$ will be glued to $T_0$.
\end{remark}

\subsubsection{Arithmetically Defined Regular Tessellations} \label{sec:constructPrinCong}

The following example shows how to construct $\MpslPrinCong{-3}{2}$ and $\MpglPrinCong{-3}{2}$. We see that $\MpslPrinCong{-3}{2}$ is a double cover of $\MpglPrinCong{-3}{2}$ (they are actually the manifold and orbifold in Figure~\ref{fig:5chainLink}).

\begin{verbatim}
>>> tessContext = TessellationContext(3,3,6,2,0)
>>> X = tessContext.PrincipalCongruenceManifold('X')
>>> len(X)
240
>>> tessContext = TessellationContext(3,3,6,2,0)
>>> Y = tessContext.PrincipalCongruenceManifold('Y')
>>> len(Y)
120
\end{verbatim}

The algorithm works as follows: To construct $\MpslPrinCongBig{D}{z}, \MpslPrinCong{D}{z},$ respectively, $\MpglPrinCong{D}{z}$, take the triangulation of $\MunivPrinCong{\pqr}{z}(n)$ from Section~\ref{sec:pythonForUniv} with $n$ large enough and fix a positively oriented simplex of it, say $\tilde{T}$. Lift $\tilde{T}$ to the universal cover $\H^3$ such that it becomes the simplex in Figure~\ref{fig:hypOct} (for $D=-3$, vertex 3 is above $(1+\zeta)/3$ instead of $(1+i)/2$). We can now label each simplex $T$ with positive orientation by a $2\times 2$ matrix that would take $\tilde{T}$ to $T$ in $\H^3$. To start, label $\tilde{T}$ by the identity. Let
$$g_P=\left(\begin{array}{cc}0 & 1\\ -1 & 1\end{array}\right),\quad
g_Q=\left(\begin{array}{cc}-1/u & 1\\ 0 & 1\end{array}\right)\quad\mbox{and}\quad
g_R=\left(\begin{array}{cc}u & 0\\ 0 & 1\end{array}\right).$$
They correspond to the rotations indicated in Figure~\ref{fig:hypOct}.
Assume $T$ is labeled by $g\in\mathrm{GL}(2,\Z[u])$. Let $T_{a_1}$ be the simplex glued to face $a_1$ of $T$ and $T_{a_1a_2}$ be the simplex glued to face $a_2$ of $T_{a_1}$. We assign the label $gg_P$ to $T_{01}$, $gg_Q$ to $T_{12}$, $gg_R$ to $T_{23}$, $gg_Pg_Q$ to $T_{02}$, $gg_Qg_R$ to $T_{13}$ and $gg_Pg_Qg_R$ to $T_{03}$. We need to identify two simplices if the images of their labels in $\mathrm{PGL}(2,\Z[u])$ differ by a matrix in $\pslPrinCongBig{D}{z}, \pslPrinCong{D}{z},$ or $\pglPrinCong{D}{z}$. We can represent the image in $\mathrm{PGL}(2,\Z[u])$ by a matrix $u^k g$ such that the determinant $\|u^k g\|$ is either 1 or $u$. Replace each label $g$ by such a matrix.

It turns out that the coefficients of the labels $g\in \mathrm{GL}(2,\Z[u])$ explode when computing the products. Hence, we instead label the simplices by pairs $(\|g\|, g \mod z)$ where $g$ is normalized as above such that $\|g\|$ is either 1 or $u$. We need to store the determinant here as well, because it is not determined by $(g\mod z)$ alone for small $z$. Now identify two simplices if their labels have
\begin{itemize}
\item the same $(g \mod z)$ up to $\pm 1$ and the same $\|g\|$ (for $\MpslPrinCongBig{D}{z}$)
\item the same $(g \mod z)$ up to $(\Z[u]/\langle z\rangle)^*$ and the same $\|g\|$ (for $\MpslPrinCong{D}{z}$)
\item the same $(g \mod z)$ up to $(\Z[u]/\langle z\rangle)^*$ (for $\MpglPrinCong{D}{z}$).
\end{itemize}

If there are open faces left, $n$ was not chosen large enough.

\subsubsection{Orbifold Detection}

For a triangulation returned by one of the above algorithms, we can check the number of simplices around each edge. If these numbers are $2p, 2q, 2r, 4, 4, 4$ for the respective edges of each simplex, it is the triangulation of a hyperbolic manifold, otherwise, an orbifold. This can be checked as follows (continuing previous example, recall that $Y$ was the orbifold and $X$ the manifold in Figure~\ref{fig:5chainLink}):

\begin{verbatim}
>>> tessContext.IsManifold(Y)
False
>>> tessContext.IsManifold(X)
True
\end{verbatim}

\subsection{Finding All Regular Tessellations Using \texttt{regularTessellations.g}}

\label{sec:allRegTessThorughGap} \label{sec:groupRep}

Recall that every regular tessellation of type $\pqr$ and cusp modulus $z$ can be obtained as quotient of $\MunivPrinCong{\pqr}{z}/N$ by a normal subgroup $N\triangleleft G$ where $G=\mathrm{Isom}^+(\MunivPrinCong{\pqr}{z})$ is the orientation-preserving symmetry group of $\MunivPrinCong{\pqr}{z}$. If $\MunivPrinCong{\pqr}{z}$ is finite-volume, we can use Gap \cite{gap:gap} to find all suitable normal subgroups and thus classify all regular tessellations in these cases. Here is an example for $\{3,3,6\}$ and $z=4$:

\begin{verbatim}
gap> Read("regularTessellations.g");
gap> G:=SymmetriesUniversalRegularTessellationPermGroup(3,3,6,4,0);
<permutation group of size 7680 with 3 generators>
gap> L:=AllRegularTessellationsFromUniversalRegularTessellation(G);;
gap> List(L,StructureDescription);
[ "1", "C2", "C4" ]
gap> IsSubgroup(L[3],L[2]);
true
\end{verbatim}

This shows that there are three manifold regular tessellations with cusp modulus $4$: $\MunivPrinCong{\{3,3,6\}}{4}$ and two extra manifolds covered by $\MunivPrinCong{\{3,3,6\}}{4}$ such that the Decktransformations are $\Z/2$, respectively, $\Z/4$. We also see that one of these extra manifolds is covered by the other one.

Here, we use the following representation for $G$ (where $z=a+bu$):
\begin{equation} \label{eqn:groupRep}
G\cong \langle P, Q, R | P^p, Q^q, R^r, (PQ)^2, (QR)^2, (PQR)^2, (QR^{r/2+1})^a (RQR^{r/2})^b\rangle.
\end{equation}
Fixing a simplex $\tilde{T}$ of $\MunivPrinCong{\pqr}{z}$, the three generators can be identified with rotations about edges of $\tilde{T}$ as shown in Figure~\ref{fig:hypOct}. Namely, a right multiplication by $P$ means we go along face 0 and 1. Analogously, $Q$ goes along face 1 and 2, and $R$ along face 2 and 3.

Given a normal subgroup $N$ of $G$, we need to check that the quotient is a manifold regular tessellation of the same cusp modulus. Let $C_1=\langle P\rangle$, $C_2=\langle PQ\rangle$, $C_3=\langle PQR\rangle$, $C_4=\langle QR\rangle$, $C_5=\langle Q\rangle$ and $C_6=\langle R\rangle$ be the cyclic subgroups of the rotations about one of the six edges of $T_0$. Let $B=\langle Q,R\rangle$ be the subgroup of elements fixing the cusp corresponding to vertex 0 of $T_0$. Then, the quotient $\MunivPrinCong{\pqr}{z}/N$ by a normal subgroup $N\triangleleft G$ is a manifold if and only if $N$ intersects each of the six $C_i$ trivially and has the same cusp modulus if and only if $N$ intersects $B$ trivially.

It should be noted that a quotient $\MunivPrinCong{\pqr}{z}/N$ can be chiral even though $\MunivPrinCong{\pqr}{z}$ is amphicheiral. Here is an example for regular tessellations of type $\{4,3,6\}$ and cusp modulus $z=2$:
\begin{verbatim}
gap> G:=SymmetriesUniversalRegularTessellationPermGroup(4,3,6,2,0);;
gap> L:=AllRegularTessellationsFromUniversalRegularTessellation(G);;
gap> IsAmphicheiralRegularTessellation(L[4],G);
false
gap> mu:=MirrorIsomorphismUniversalRegularTessellation(G);;
gap> Image(mu,L[4])=L[5];
true
\end{verbatim}

We see that the fourth regular tessellation in the list is chiral and that its mirror image is the fifth regular tessellation in the list (Table~\ref{table:FinCat} lists them as $\mathcal{Z}_0$ and $\bar{\mathcal{Z}}_0$).

To detect this in general when $\MunivPrinCong{\pqr}{z}$ is amphicheiral, let $\mu$ be the group automorphism obtained by conjugating $G=\mathrm{Isom}^+(\MunivPrinCong{\pqr}{z})$ with an orientation-reversing symmetry:
$$\mu: G\to G, P\mapsto Q^{-1} P^{-1} Q, Q\mapsto Q^{-1}, R\mapsto R^{-1}.$$
Now the mirror image of a quotient $\MunivPrinCong{\pqr}{z}/N$ is given by $\MunivPrinCong{\pqr}{z}/\mu(N)$ and the quotient is amphicheiral if and only if $N=\mu(N)$.

\newpage

\subsection{Proving a Manifold Is a Link Complement} \label{sec:howToFindDehnFillings}

\subsubsection{Link Complement Certificates}

For 19 of the 21 potential finite-volume manifold universal regular tessellations $\MunivPrinCong{\pqr}{z}$ (those not marked with a ``$*$'' or a ``?'' in Table ~\ref{table:finiteVolUnivReg}), we provide SnapPy files named \texttt{\dots{}with\_meridians.trig} certifying that $\MunivPrinCong{\pqr}{z}$ is indeed a link complement. Except for $\MunivPrinCong{\{5,3,6\}}{2}$, each respective file contains a triangulation that is homeomorphic to $\MunivPrinCong{\pqr}{z}$. The peripheral curves saved in the file are such that $(1,0)$-Dehn-filling along each cusp yields a manifold with trivial fundamental group and, hence, homeomorphic to $S^3$ by Perelman's Theorem.

For $\MunivPrinCong{\{5,3,6\}}{2}$, we use the quotient $\Mmfd=\MunivPrinCong{\{5,3,6\}}{2}/H$ by a suitable subgroup $H\cong\Z/15$ which has the same cusp modulus.
$(1,0)$-Dehn-filling $\Mmfd$ now yields a manifold $L$ with fundamental group $\Z/15$ (actually a lens space by Geometrization). Thus, lifting the embedded $\Mmfd\subset L$ to the universal cover of $L$ gives a link complement $\tilde{\Mmfd}\subset S^3$. We need to verify that $\tilde{\Mmfd}$ is indeed $\MunivPrinCong{\{5,3,6\}}{2}$. It is enough to show that $\tilde{\Mmfd}\to\Mmfd$ is a cuspidal covering map as defined in Section~\ref{sec:cuspidalCovers}. The cuspidal homology $H^{cusp}_1(\Mmfd)\cong\Z/15$ is equal to $H_1(L)\cong\pi_1(L)$ of the filled manifold $L$. So, $H_1(\partial\Mmfd)$ dies in $H_1(L)$ and a cusp of $\Mmfd$ lifts to 15 disjoint copies in $\tilde{\Mmfd}$.

The reader can thus prove that a universal regular tessellation is a link complement by:
\begin{enumerate}
\item Following the steps in \ref{sec:pythonForUniv} to write $\MunivPrinCong{\pqr}{z}$ to a SnapPy file (for $\Mmfd=\MunivPrinCong{\{5,3,6\}}{2}/H$, using \texttt{tess=DodecahedralUniversalQuotientedByC15()} which also verifies that the quotient has the same cusp modulus).
\item Comparing the resulting manifold with the provided file using SnapPy \texttt{Manifold}'s \texttt{is\_isometric\_to}. (In case of $\MunivPrinCong{\{5,3,6\}}{2}$, also verifying $H_1(\Mmfd)=\Z^{40}\oplus\Z/15$ and that $\Mmfd$ has 40 cusps).
\item Dehn-filling the provided manifold and check that $\pi_1$ is trivial (in case of $\MunivPrinCong{\{5,3,6\}}{2}$, $\pi_1\cong\Z/15$). This has been automated in the script\\ \texttt{universalRegularTessellationLinkComplementsProofs.py}.
\end{enumerate}

\subsubsection{Method Used To Find the Certificates} \label{subsubsec:findingSlopes}

It turns out that the existing methods to find exceptional slopes did not suffice (see discussion section), so we just briefly sketch the procedure we used to find the right slopes for the Dehn-fillings of a $\MunivPrinCong{\pqr}{z}$. We iterated the following steps in SnapPy until there were only a few unfilled cusps left:
\begin{enumerate}
\item Pick the next unfilled cusp of the manifold.
\item Among the slopes (1,0), (0,1), (1,1), (-1,1), fill along the one resulting in the least ``volume'' as reported by SnapPy's \texttt{volume} and for which $H_1=\Z^c$ where $c$ is the number of unfilled cusps.
\item Every couple of iterations, replace the deformed manifold with the filled manifold (from \texttt{filled\_triangulation()})
\end{enumerate}

For the last few unfilled cusps, we use a brute force search among all combinations of the above slopes until we find one for which SnapPy computes a presentation of the fundamental group with no generators. Once suitable slopes are found, the above certificates can be created using SnapPy's \texttt{set\_peripheral\_curves}.

In case of $\MunivPrinCong{\pqr}{z}$, we use Gap to find cyclic subgroups $H$ of $G$ from presentation~\eqref{eqn:groupRep} that intersect all conjugates of the subgroups $C_i$ and $B$ trivially. Given such an $H$, we find a word representing the generator of $H$ and use it to identify a positively oriented tetrahedron in $\MunivPrinCong{\{5,3,6\}}{2}$ with the image of the tetrahedron when applying the word to it (see~\eqref{eqn:groupRep} and the following comment to see how a word acts on the tetrahedra). The result is the quotient space $\Mmfd=\MunivPrinCong{\{5,3,6\}}{2}/H$. For one of these cyclic subgroups $H$, applying techniques analogous to the above to $\Mmfd$ yields the certificate described earlier.

\section{Proof of Infinite Universal Regular Tessellation for Large-Enough Cusp Modulus} \label{sec:proofInfLarge}

In this section, we always assume that $z$ has combinatorial length 6 or greater where combinatorial length is defined as:

\begin{definition}
The combinatorial length $\|z\|$ of $z\in \calO_D$ is the minimal number $n$ of units $z_1,\dots,z_n\in\calO_D^*$ such that $z=z_1+\dots+z_n.$
\end{definition}

The combinatorial length is a norm on $\calO_D$, invariant under complex conjugation and multiplication by a unit. If $z$ is in canonical form $a+bu$, then the combinatorial length $\|z\|$ is just $a+b$.
\begin{theorem}\label{thm:bigMod}
Let $z\in\calO_{-3}$ and $\|z\|\geq 6$. Then, $\MunivPrinCong{\{3,3,6\}}{z}$, $\MunivPrinCong{\{4,3,6\}}{z}$, and $\MunivPrinCong{\{5,3,6\}}{z}$ are infinite volume. Let $z=a+bi\in\calO_{-4}$ and $\|z\|\geq 9$ or $a\geq 6$, then $\MunivPrinCong{\{3,4,4\}}{z}$ is infinite volume.
\end{theorem}

The rest of this section is devoted to the proof of this theorem which is split into various cases with each case being an immediate consequence of Lemma~\ref{lemma:ConstructionTerminates} and one of Lemmas~\ref{lemma:case336}, \ref{lemma:case436}, \ref{lemma:case536}, and \ref{lemma:case344}. Recall the construction in Definition~\ref{def:UnivRegTessAlgo}.

\begin{figure}[h]
\begin{center}
\scalebox{0.32}{\input{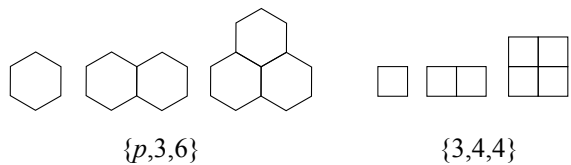}}
\end{center}
\caption{\label{fig:clusters}Small $(p-1)$-clusters.}
\end{figure}

\begin{definition}
Split the surface $\partial \MunivPrinCong{\pqr}{z}(n)$ along all edges labeled by $p-2$ or less. We call the resulting connected components $(p-1)$-clusters. We call a $(p-1)$-cluster small if it is one of the clusters shown in Figure~\ref{fig:clusters}.
\end{definition}

We can align a nanotube and a small $(p-1)$-cluster such that each face of the cluster goes to a face of the nanotube and we can attach the nanotube by gluing along those faces. Thus, we can eliminate the need to ever identify nanotubes by simplifying the construction as follows:

\begin{definition}[Simplified Construction]
To obtain $\barMunivPrinCong{\pqr}{z}(n+1)$ from $\MunivPrinCong{\pqr}{z}(n)$:
\begin{enumerate}
\item Attach one new nanotube for each $(p-1)$-cluster along all the faces of the $(p-1)$-cluster simultaneously.
\item If there is an edge $e$ with $p$ nanotubes around $e$ and two open faces adjacent to $e$, we need to glue the two faces. Repeat until there is no edge left for which such a gluing is necessary.
\end{enumerate}
If $\MunivPrinCong{\pqr}{z}(n)$ has a $(p-1)$-cluster that is not small or we encounter an edge $e$ with more than $p$ nanotubes around $e$ or an edge cycle of length different from $p$ during the construction, we say that the construction of $\barMunivPrinCong{\pqr}{z}(n+1)$ failed.
\end{definition}

Unless the simplified construction $\barMunivPrinCong{\pqr}{z}(n+1)$ fails, it is equal to $\MunivPrinCong{\pqr}{z}(n+1)$. In this case, the map $\MunivPrinCong{\pqr}{z}(n)\to\MunivPrinCong{\pqr}{z}(n+1)$ is also an embedding as the construction of $\barMunivPrinCong{\pqr}{z}(n+1)$ only glues but never identifies. Futhermore, the edge-labeled surface $\partial\MunivPrinCong{\pqr}{z}(n)$ determines $\partial\MunivPrinCong{\pqr}{z}(n+1)$ completely in this case.

\begin{figure}[h]
\begin{center}
\scalebox{0.23}{\input{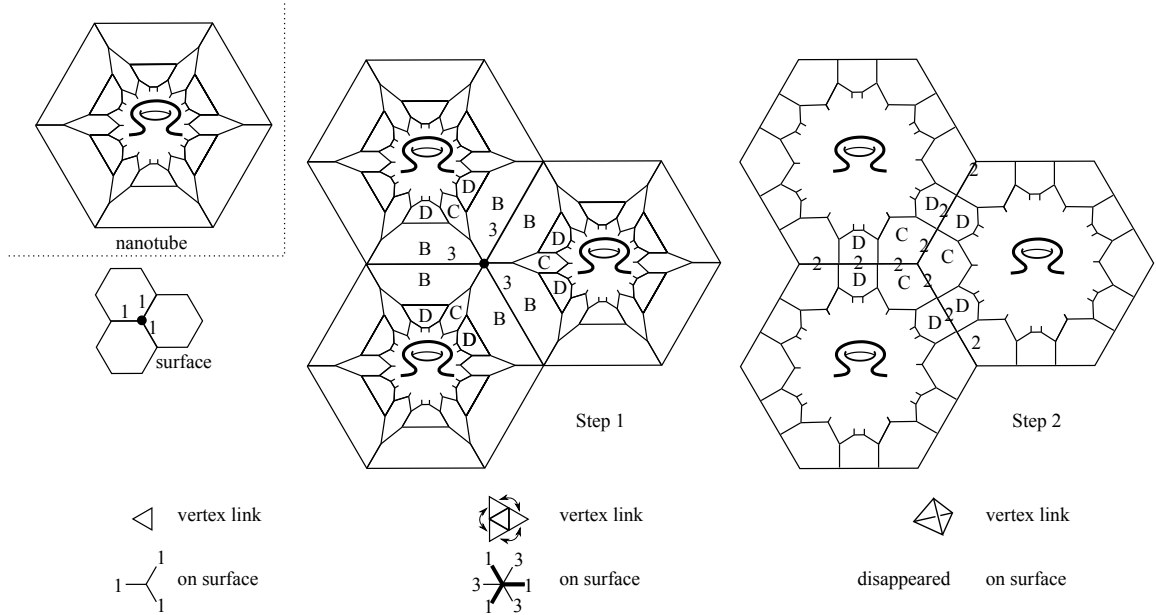}}
\end{center}
\caption{\label{fig:attachingProcess}Graphical construction $\barMunivPrinCong{\{3,3,6\}}{z}(1)$ shown left to right using a stretched nanotube shown on the top left. The evolution of the link of the dotted vertex is shown below. A bold line indicates this is a ``new'' edge.}
\end{figure}

Figure~\ref{fig:attachingProcess} shows a way of visualizing the construction, here of $\barMunivPrinCong{\{3,3,6\}}{z}(1)$. We can take a nanotube and stretch one hexagon to be very large and facing away from us as shown on the top left. Imagine a small piece of the surface $\partial\MunivPrinCong{\{3,3,6\}}{z}(0)$ which consists of hexagons and edges labeled by 1. In step 1, we glue a copy of the stretched nanotube on top of each hexagon. The old edges will now be labeled with 3, hence step 2 will glue them up resulting in the next picture. Note that the faces marked $C$ form a $2$-cluster of 3 hexagons and the faces marked $D$ a $2$-cluster of 2 hexagons. All other edges will be labeled by 1.

\begin{remark} \label{remark:vertexLinkEvolution}
We can also draw pictures to study the evolution of a vertex link as in the bottom of the figure. The {\em evolution of a vertex link} of $\MunivPrinCong{\pqr}{z}(n+1)$ during the construction of $\barMunivPrinCong{\pqr}{z}(n+1)$ is as follows: take the boundary of the vertex link, split it at the points where $p-2$ or fewer $q$-gons meet. Step 1: Attach a new $q$-gon along each component. Step 2: if a point now has $p$ adjacent polygons, glue the two open sides. As we will see later, a vertex link will always be a subcomplex of the polyhedron $\{q,p\}$, the dual of the Platonic solid of the regular tessellation. It eventually closes up to the polyhedron $\{q,p\}$ and thus will disappear from the surface $\partial\MunivPrinCong{\pqr}{z}(n)$ as $n$ grows.
\end{remark}

\subsection{The Cases $\{p,3,6\}$}

\begin{figure}[h]
\begin{center}
\scalebox{0.41}{\input{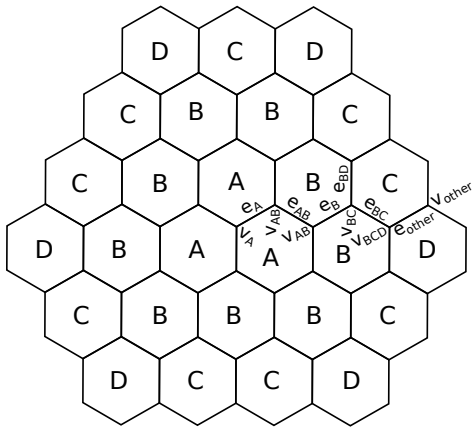}}
\end{center}
\caption{\label{fig:gluingPattern336}\label{fig:labelsOnNanotube}Labels of faces, edges, and vertices of a nanotube's surface.}
\end{figure}

Given a newly attached attached nanotube, we label the hexagons as follows:
\begin{itemize}
\item $A$: faces along which the nanotube was attached.
\item $B$: neighbors of $A$-faces
\item $C$: faces neighboring two $B$-faces
\item $D$: faces neighboring one $B$-face
\item ``other'' faces: remaining faces
\end{itemize}
We label an edge or vertex by $e$, respectively, $v$ and a subscript indicating all classes of faces that are adjacent to the edge or vertex. An edge or vertex that is not adjacent to any $A$- or $B$-face is called ``other'' edge $e_{\mathrm{other}}$, respectively, ``other'' vertex $v_{\mathrm{other}}$. This is shown for a nanotube attached along a $(p-1)$-cluster of 3 hexagons in Figure~\ref{fig:labelsOnNanotube}.

\begin{lemma}\label{lemma:noCollision}
If $\|z\|\geq 6$, the pattern of $A$-, $B$-, $C$-, and $D$-hexagons in Figure~\ref{fig:labelsOnNanotube} and its analogues for smaller $(p-1)$-clusters embeds into the nanotube's surface $T^*_z$.
\end{lemma}

\begin{proof}
\begin{figure}[h]
\scalebox{0.20}{\input{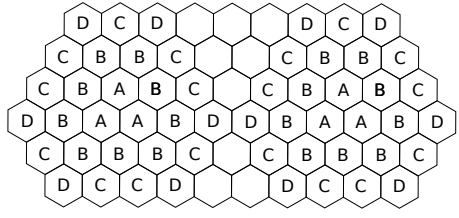}}
\caption{\label{fig:336NoCollision}The pattern from Figure~\ref{fig:labelsOnNanotube} does not overlap with a copy of it self when translated by $t=6$.}
\end{figure}

We work in the universal covering space which is the complex plane tessellated by honeycombs: we need to show that the pattern does not collide with a copy of itself in the complex plane when translated by an element $t\in \langle z\rangle\setminus 0$. Note that $\|t\|\geq 6$. By symmetry, we ca assume that $t$ is in canonical form. Figure~\ref{fig:336NoCollision} and similar diagrams for $t=5+\zeta$, $4+2\zeta$, $3+3\zeta$ show that there is no collision when $\|t\|=6$. Larger value of $\|t\|$ ensure this as well.
\end{proof}

\newpage
Assuming that $\|z\|\geq 6$ so the lemma holds true, we make the following observations about the construction of $\barMunivPrinCong{\{p,3,6\}}{z}(n+1)$:
\begin{enumerate}[(i)]
\item The following vertices and edges were already in $\MunivPrinCong{\{p,3,6\}}{z}(n)$: $v_A$, $v_{AB}$, $e_{A}$, $e_{AB}$. We thus call them ``old''.
\item During step 1, $A$-faces, $v_{A}$ and $e_{A}$ are always glued up and disappear from the surface $\partial\barMunivPrinCong{\{p,3,6\}}{z}(n+1)$.
\item The label of an edge $e_{AB}$ is two plus the original label of the corresponding edge in $\partial\MunivPrinCong{\{p,3,6\}}{z}(n)$. This label might potentially be $p$ and thus the edge might glue up causing the adjacent $B$-face to be glued up to a $B$-face of another nanotube.
\item This might cause an edge $e_{B}$ to have an edge label greater 1 and to even glue up.
\item An edge $e_{BC}$, respectively, $e_{BD}$ will be labeled with $1$ or $2$ depending on whether the $B$-face was glued.
\item Any ``other'' edge has label 1.
\item This leaves edge $e_{B}$ as only potential edge having more than $p$ nanotubes about it causing the simplified construction of $\barMunivPrinCong{\{p,3,6\}}{z}(n+1)$ to fail.\label{item:failBecauseOfTooManyNanoTubes}
\item As the vertex $v_{AB}$ is ``old'', its vertex link is obtained as described in Remark~\ref{remark:vertexLinkEvolution}.
\item A ``new'' vertex $v_{BC}$ has as vertex link a number of triangles sharing a vertex. This number is the edge label of the adjacent $e_{B}$ edge after step 2. If it is $p$, the vertex link is a cycle of triangles.\label{item:newVertexLink}
\item A ``new'' vertex $v_{BCD}$ has vertex link one or two triangles.
\item Any ``other'' vertex is ``new'' and has vertex link one triangle.
\item If $p=3$, the $B$-faces always glue up. When looking for faces that could potentially form too large $2$-clusters on the surface of $\barMunivPrinCong{\{p,3,6\}}{z}(n+1)$, we only need to consider $C$- and $D$-faces as they are the only unglued faces with edges labeled with 2.\label{item:failBecauseOfTooLargeTwoCluster}
\item If $p>3$, the only edges with new label larger than 2 are $e_{AB}$ and $e_{B}$. Thus we only need to consider $B$-faces as potentially forming too large $p-1$ clusters. \label{item:failBecauseOfTooLargeCluster}
\end{enumerate}


\begin{lemma} \label{lemma:case336}
Let $\|z\|\geq 6$. The following hold for $\MunivPrinCong{\{3,3,6\}}{z}(n)$:
\begin{itemize}
\item $\partial \MunivPrinCong{\{3,3,6\}}{z}(n)$ is non-empty.
\item Every vertex link of $\MunivPrinCong{\{3,3,6\}}{z}(n)$ is a subcomplex of the tetrahedron.
\item Every $2$-cluster is small.
\end{itemize}
Thus, the simplified construction $\barMunivPrinCong{\{3,3,6\}}{z}(n+1)\cong\MunivPrinCong{\{3,3,6\}}{z}(n+1)$ is well-defined.
\end{lemma}

\begin{proof} The three properties of the lemma are obviously true for $\MunivPrinCong{\{3,3,6\}}{z}(0)$. Assume they are true for $n$. 

The first property holds for $n+1$ because a $C$-face will be unglued.

\begin{figure}[h]
\begin{center}
\scalebox{0.6}{\input{figures_gen/tetrahedralVertexLinks.tex}}
\end{center}
\caption{\label{fig:tetrahedralVertexLinks}\label{fig:tetrahedralVertexLinkEvolution}Vertex link evolution for $\MunivPrinCong{\{3,3,6\}}{z}(n)$.}
\end{figure}

Figure~\ref{fig:tetrahedralVertexLinkEvolution} shows that the vertex link of any ``old'' vertex becomes a tetrahedron. By observation~\eqref{item:newVertexLink}, the vertex link of a ``new'' vertex is a subcomplex of a tetrahedron as shown on the left column of Figure~\ref{fig:tetrahedralVertexLinkEvolution}. Thus, the second property holds. Since every ``old'' vertex immediately closes up to a tetrahedron, every $e_{B}$ always closes up to a cycle with $3$ nanotubes around it. Thus, by observation~\eqref{item:failBecauseOfTooManyNanoTubes} above, the simplified construction $\barMunivPrinCong{\{3,3,6\}}{z}(n+1)$ is well-defined and equal to $\MunivPrinCong{\{3,3,6\}}{z}(n+1)$.

It is left to show that every $2$ cluster of $\partial \barMunivPrinCong{\{3,3,6\}}{z}(n+1)$ is small. Regarding observation~\eqref{item:failBecauseOfTooLargeTwoCluster}, we only need to look at $C$- and $D$-faces. Since a $D$-face has only one edge with label 2, namely $e_{BD}$, a 2-cluster of more than two hexagons has to contain a $C$-face. Such a $C$-face forms a 2-cluster of three hexagons together with $C$-faces of other nanotubes. To see this, look at the vertex link of $v_{BC}$ which by observation~\eqref{item:newVertexLink} consists of a cycle of 3 triangles.
\end{proof}

\begin{figure}[h]
\begin{center}
\scalebox{0.54}{\input{figures_gen/octahedralVertexLinks.tex}}
\end{center}
\caption{\label{fig:octahedralVertexLinkEvolution} Vertex link evolution for $\MunivPrinCong{\{4,3,6\}}{z}(n)$.}
\end{figure}


\begin{lemma} \label{lemma:case436}
Let $\|z\|\geq 6$. The following hold for $\MunivPrinCong{\{4,3,6\}}{z}(n)$:
\begin{itemize}
\item $\partial\MunivPrinCong{\{4,3,6\}}{z}(n)$ is non-empty.
\item Every vertex link of a vertex in $\partial\MunivPrinCong{\{4,3,6\}}{z}(n)$ is as shown in Figure~\ref{fig:octahedralVertexLinkEvolution}, in particular all edges of $\partial\MunivPrinCong{\{4,3,6\}}{z}(n)$ are labeled by 1 or 3.
\item Every $3$-cluster is small.
\end{itemize}
Thus, the simplified construction $\barMunivPrinCong{\{4,3,6\}}{z}(n+1)\cong\MunivPrinCong{\{4,3,6\}}{z}(n+1)$ is well-defined.
\end{lemma}

\begin{proof}
The three properties of the lemma are obviously true for $\MunivPrinCong{\{4,3,6\}}{z}(0)$. Assume it is true for $n$.

The first property holds for $n+1$ because a $C$-face will be unglued.

Figure~\ref{fig:octahedralVertexLinkEvolution} shows the evolution of a vertex link, so since any ``old'' vertex had a vertex link shown in the figure for $n$ by assumption, it will have again a vertex link shown in the figure for $n+1$. Since no edge had label 2, no $B$-face will be glued. In particular, the label of edges $e_{B}$, $e_{BC}$, and $e_{BD}$ will be 1, so the vertex link of any ``new'' vertex is just a triangle. Thus, we have proven the second property and also that the simplified construction $\barMunivPrinCong{\{4,3,6\}}{z}(n+1)$ is equal to $\MunivPrinCong{\{4,3,6\}}{z}(n+1)$.

It also follows that $3$-clusters of more than one hexagon can only be formed by $B$-faces. $3$-clusters of more than two hexagons have to contain a $B$-face touching two $A$-faces. Looking at the vertex links of the vertex that the $B$-face shares with the two $A$-faces, we see that we get a $3$-cluster of three hexagons.
\end{proof}

\begin{lemma} \label{lemma:case536}
Let $\|z\|\geq 6$. The following hold for $\MunivPrinCong{\{5,3,6\}}{z}(n)$:
\begin{itemize}
\item $\partial\MunivPrinCong{\{5,3,6\}}{z}(n)$ is non-empty.
\item Every vertex link of a vertex in $\partial\MunivPrinCong{\{5,3,6\}}{z}(n)$ is as shown in Figure~\ref{fig:icosahedralVertexLinkEvolution}.
\item Every $4$-cluster is small.
\end{itemize}
Thus, the simplified construction $\barMunivPrinCong{\{5,3,6\}}{z}(n+1)\cong\MunivPrinCong{\{5,3,6\}}{z}(n+1)$ is well-defined.
\end{lemma}

\begin{figure}[h]
\begin{center}
\scalebox{0.39}{\input{figures_gen/icosahedralVertexLink.tex}}
\end{center}
\caption{\label{fig:icosahedralVertexLinkEvolution} Vertex link evolution for $\MunivPrinCong{\{5,3,6\}}{z}(n)$.}
\end{figure}

\begin{proof}
The three properties of the lemma are obviously true for $\MunivPrinCong{\{5,3,6\}}{z}(0)$. Assume it is true for $n$.

The first property holds for $n+1$ because a $C$-face will be unglued.

Figure~\ref{fig:icosahedralVertexLinkEvolution} shows the evolution of a vertex link. An ``old'' vertex had one such vertex link for $n$, so it will have the next vertex link in the figure for $n+1$. By observation~\eqref{item:newVertexLink}, a ``new'' vertex will have one of the five vertex links on the left of the figure. So, the second property holds and we also never run into a case of an edge having more than $5$ nanotubes about it, so $\barMunivPrinCong{\{5,3,6\}}{z}(n+1)$ is well-defined.

It remains to show that every $4$-cluster is small. By observation~\eqref{item:failBecauseOfTooLargeCluster}, it is enough to look at the unglued $B$-faces. The unglued $B$-faces are those faces of $\partial\barMunivPrinCong{\{5,3,6\}}{z}(n+1)$ that are adjacent to some ``old'' vertex $v_{AB}$. Figure~\ref{fig:icosahedralVertexLinkEvolution} shows what $\partial\barMunivPrinCong{\{5,3,6\}}{z}(n+1)$ looks like near an ``old'' vertex. The ``new'' edges $e_{B}$ are marked in bold to distinguish them from the ``old'' edges $e_{AB}$.

We need to show that for each unglued $B$-face, the cluster we obtain by following all its edges with label 4 is small. Recall that a $B$-face can potentially have up to four edges with label 4: two ``new'' edges $e_B$ and up to two ``old'' edges $e_{AB}$ meeting at an ``old'' vertex.

First, we consider only $B$-faces where an ``old'' edge is labeled by 4 but no ``new'' edge (bold) is labeled by 4, so a 4-cluster has to form around an ``old'' vertex. As we see in Figure~\ref{fig:icosahedralVertexLinkEvolution}, these clusters are all small.

Now, we consider the case where a ``new'' $e_B$ edge has label 4. We find exactly one vertex link (third row, fourth column) where this is the case. The diagram left to it shows what $\partial\MunivPrinCong{\{5,3,6\}}{z}(n)$ looks near that vertex. When $\partial\MunivPrinCong{\{5,3,6\}}{z}(n)$ was constructed, the two edges labeled with 2 were ``new'' edges $e_B$, hence their other end points were ``new'' vertices, so by observation~\eqref{item:newVertexLink}, we know that their vertex link consists of two triangles (new vertex in the second row of Figure~\ref{fig:icosahedralVertexLinkEvolution}). Thus, we can draw $\partial\MunivPrinCong{\{5,3,6\}}{z}(n)$ near these three vertices and obtain the diagram on the left of Figure~\ref{fig:icosahedralVertexLinkEvolutionSpecialCase}. Similarly to Figure~\ref{fig:attachingProcess}, we can perform the construction of $\barMunivPrinCong{\{5,3,6\}}{z}(n+1)$, this time stretching one of nanotubes such that two $A$-faces are facing away from us. We see that in step 2, six of the $B$-faces glue up and the remaining three $B$-faces touching the central vertex form a small $4$-cluster.

\begin{figure}[h]
\begin{center}
\scalebox{0.26}{\input{figures_gen/icosahedralVertexLinkSpecialCase.tex}}
\end{center}
\caption{\label{fig:icosahedralVertexLinkEvolutionSpecialCase}Construction of $\partial\barMunivPrinCong{\{5,3,6\}}{z}(n+1)$ for special case.}
\end{figure}

\end{proof}

\subsection{The Case $\{3,4,4\}$}

\begin{lemma}
Let $\|z\|\geq 9$ or $z=a+bi$ with $a\geq 6$. The following hold for $\MunivPrinCong{\{3,4,4\}}{z}(n)$:
\begin{itemize}
\item $\partial\MunivPrinCong{\{3,4,4\}}{z}(n)$ is non-empty.
\item Every vertex of $\partial\MunivPrinCong{\{3,4,4\}}{z}(n)$ has one of the vertex links shown in Figure~\ref{fig:cubicalVertexLinkEvolution}.
\item Every $2$-cluster is small. Furthermore, most vertices of a $2$-cluster have trivial vertex link as indicated in Figure~\ref{fig:large344Patterns}.
\end{itemize}
Thus, the simplified construction $\barMunivPrinCong{\{3,4,4\}}{z}(n)\cong \MunivPrinCong{\{3,4,4\}}{z}(n)$ is well-defined.
\label{lemma:case344}
\end{lemma}

\begin{figure}[h]
\begin{center}
\scalebox{0.8}{\input{figures_gen/cubicalVertexLinks.tex}}
\end{center}
\caption{\label{fig:cubicalVertexLinkEvolution} Vertex link evolution for $\MunivPrinCong{\{3,4,4\}}{z}(n)$.}
\end{figure}

\begin{figure}[h]
\begin{center}
\scalebox{1.1}{\input{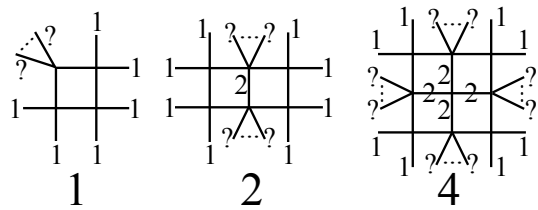}}
\end{center}
\caption{\label{fig:large344Patterns}Neighborhoods of $2$-clusters.}
\end{figure}

\begin{proof}
The claims in the lemma are obviously true for $n=0$. Assume they are true for $n$. Similar to Figure~\ref{fig:attachingProcess}, we can graphically construct $\partial\barMunivPrinCong{\{3,4,4\}}{z}(n+1)$ as shown in Figure~\ref{fig:344Pattern1Attach}. 

\begin{figure}[h]
\begin{center}
\scalebox{0.125}{\input{figures_gen/InvertedSquareNanotube.tex}}
\end{center}
\caption{\label{fig:344Pattern1Attach}Graphical construction of $\partial\barMunivPrinCong{\{3,4,4\}}{z}(n+1)$.}
\end{figure}

We label the faces of a newly attached nanotube:
\begin{itemize}
\item $A$: faces along which the nanotube was attached.
\item $B$: a face that shares an edge with $A$ or a vertex with $A$ that had a non-trivial vertex link in $\partial\MunivPrinCong{\{3,4,4\}}{z}(n)$.
\item $C$: faces neighboring two $B$-faces
\item $D$: faces neighboring one $B$-face
\item $E$: faces touching a $B$-face in a vertex.
\item ``other'' faces: remaining faces
\end{itemize}

We call a vertex touching an $A$-face ``old'' and all other vertices ``new''. Assume that the pattern of $A$-, $B$-, $C$-, $D$-, and $E$-faces embeds into the surface of a  nanotube. Then the vertex links evolve as shown in Figure~\ref{fig:cubicalVertexLinkEvolution}. In particular, every non-trivial vertex link closes up to a cube and every $B$-face glues up. The new $\partial\barMunivPrinCong{\{3,4,4\}}{z}(n+1)$ thus consists only of $C$-, $D$-, $E$-, and ``other'' faces. An edge of $\partial\barMunivPrinCong{\{3,4,4\}}{z}(n+1)$ has label 2 if and only if it touched a $B$-face. Similarly, a ``new'' vertex has non-trivial vertex link if and only if it is adjacent to a $B$-face. A $C$-face thus has two edges with label 2. Four $C$-faces form a 2-cluster around an ``old'' vertex. The vertex of $C$ opposite to the ``old'' vertex touches no $B$-face and will have trivial vertex link. Thus we get the pattern on the right of Figure~\ref{fig:large344Patterns}. Similarly, a $D$-face has exactly one edge with label 2 and thus forms a cluster with exactly one other $D$-face. The two vertices where these two $D$-faces meet have non-trivial vertex link but all other vertices of the $D$-face have trivial vertex link so we get the pattern in the middle of Figure~\ref{fig:large344Patterns}. An $E$-face has no edge with label 2. Only one of its vertices touches a $B$-face, so three of its vertices have trivial vertex link. Every ``other'' face forms a 2-cluster by itself with all vertex links trivial. Thus we get the pattern on the left in Figure~\ref{fig:large344Patterns} for $E$- and ``other'' faces.

It remains to show that the pattern of $A$-, $B$-, $C$-, $D$-, and $E$-faces embeds into the surface of a nanotube for each 2-cluster as shown in Figure~\ref{fig:large344Patterns}. These patterns are shown in Figure~\ref{fig:FacePattern334}. We distinguish between two cases for 2-clusters of one square depending on whether three or four vertices of the square have trivial vertex link.

\begin{figure}[h]
\begin{center}
\scalebox{0.3}{\input{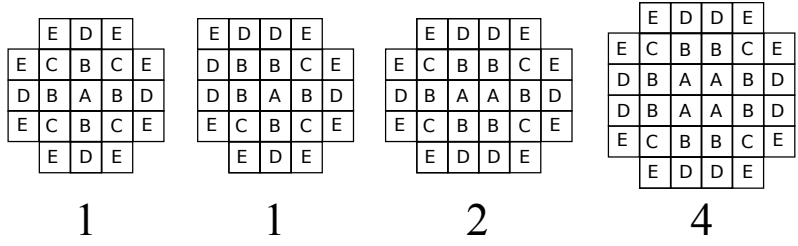}}
\end{center}
\caption{\label{fig:FacePattern334}Labels of faces of a nanotubes attached along a $2$-cluster.}
\end{figure}

\begin{figure}[h]
\begin{center}
\scalebox{0.24}{\input{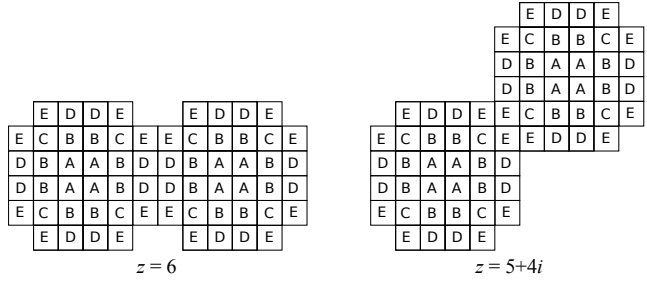}}
\end{center}
\caption{\label{fig:noCollisionSquares}No collision of patterns for $z=6$ and $z=5+4i$.}
\end{figure}

Similarly to the proof of Lemma~\ref{lemma:noCollision}, it is enough to show that two copies of these patterns do not collide when translated by $t\in\langle z\rangle\setminus 0$. The left configuration in Figure~\ref{fig:noCollisionSquares} shows that they do not collide when translated by $z=a+bi$ with $a\geq 6$. The right configuration shows that they do not collide when translated $z=5+4i$ and that we can furthermore move one copy by an integral multiple of $i-1$ without collision. Thus $a\geq 6$ or $\|z\| \geq 9$ suffices.
\end{proof}

\section{Cuspidal Covers and Homology} \label{sec:cuspidalCovers}

\begin{definition}\label{def:cuspidalCover}
A cuspidal covering space $\tilde{\Mmfd}\to\Mmfd$ of a cusped 3-manifold $\Mmfd$ is a covering space such that each cusp has a neighborhood $\mathcal{N}$ whose preimage consists of disjoint copies of $\mathcal{N}$.
\end{definition}

We can extend the universal property of $\MunivPrinCong{\pqr}{z}$ to the larger class of tessellations with well-defined cusp modulus: Let $\Mmfd$ be a (not necessarily regular) tessellation by ideal Platonic solids. We say that $\Mmfd$ has {\bf well-defined cusp modulus $z$} if the tessellation induced on each cusp is a regular tessellation $T_z$ and characterized by the same $z$ (up to a unit) for every cusp. Note that $\tetGroup{\pqr}$ acts transitively on the ideal vertices of the tessellation $\pqr$ and thus can take any cusp to the cusp corresponding to $\infty$. Thus, we can also define it algebraically:

\begin{definition}
Let $\Mmfd$ be a cusped orientable 3-manifold such that $\Mmfd=\H^3/\mfd$ where $M$ is a torsion-free subgroup of $\tetGroup{\pqr}$. Then $\Mmfd$ has well-defined cusp modulus $z$ if $(gMg^{-1})\cap P=P_z$ for every $g\in\tetGroup{\pqr}$.
\end{definition}

\begin{lemma}
Let $\Mmfd$ be a tessellation with well-defined cusp modulus $z$ and $\tilde{\Mmfd}$ be a cuspidal covering space. Then the tessellation $\tilde{\Mmfd}$ has well-defined cusp modulus $z$.
\end{lemma}

\begin{proof}
Follows from the definitions.
\end{proof}

\begin{lemma}
Let $\Mmfd$ be a tessellation of type $\pqr$ with well-defined cusp modulus $z$. Then there is a cuspidal covering map $\MunivPrinCong{\pqr}{z}\to\Mmfd$. 
\end{lemma}

\begin{proof}
As a group, $\univPrinCong{\pqr}{z}$ is generated by all elements $gp_z g^{-1}$, thus it must be contained in $\mfd$.
\end{proof}

For the rest of this chapter, we will use $\Mmfd$ interchangeably to denote a cusped hyperbolic manifold as well as the manifold with boundary whose interior is the cusped hyperbolic manifold. Let $i:\partial\Mmfd\to\Mmfd$ be the inclusion of the boundary. Let $H_1^{cusp}(\Mmfd)=H_1(\Mmfd)/\mathrm{Im}(i_*)$.

\begin{lemma}
Connected cuspidal Abelian covering spaces $\tilde{\Mmfd}\to\Mmfd$ correspond to epimorphisms $\rho: H_1^{cusp}(\Mmfd)\to G$. The number of sheets of the covering space is $|G|$.
\end{lemma} 

\begin{proof}
Connected regular covering spaces $\tilde{\Mmfd}\to\Mmfd$ correspond to epimorphisms $\rho:\pi_1(\Mmfd)\to G$. Such a cover is cuspidal if the image of every peripheral curve is trivial and is Abelian if $\rho$ factors through the homology $H_1(\Mmfd)$. The peripheral curves go to $\mathrm{Im}(i_*)$ in $H_1^{cusp}(\Mmfd)$. Hence, if a cover is both cuspidal and Abelian, $\rho$ factors through $H_1^{cusp}(\Mmfd)$.
\end{proof}

The following lemma is useful to compute $H_1^{cusp}(\Mmfd)$, especially in SnapPy, which computes only $H_1(\Mmfd)$:
\begin{lemma} \label{lemma:computeCuspidalHomology}
Let $\Mmfd$ be a cusped hyperbolic 3-manifold with $c$ cusps. Then $H_1(\Mmfd)\cong\Z^c\oplus H_1^{cusp}(\Mmfd)$.
\end{lemma}

\begin{proof}
For rational coefficients, the result follows from the Half-Lives-Half-Dies Theorem \cite[Chapter~VI,~Theorem~10.4]{bredon:top_and_geo} stating that $$\dim \ker(i_*:H_1(\partial\Mmfd;\Q)\to H_1(\Mmfd;\Q))=\dim H^1(\partial \Mmfd;\Q)/2=c.$$ We have $\dim H_1(\partial\Mmfd;\Q)=2c$, so $\dim \mathrm{Im}(i_*:H_1(\partial\Mmfd;\Q)\to H_1(\Mmfd;\Q))=c$, so $\dim H_1(\Mmfd)\otimes\Q = c + \dim(H_1^{cusp}(\Mmfd)\otimes\Q)$.

It is left to show that the torsion of $H_1(\Mmfd)$ matches that of $H_1^{cusp}(\Mmfd)$. The universal coefficient theorem for cohomology states that
$$0\to\mathrm{Ext}(H_1(\Mmfd),\Z)\to H^2(\Mmfd)\to\mathrm{Hom}(H_2(\Mmfd),\Z)\to 0,$$
so the torsion of $H_1(\Mmfd)$ given by $\mathrm{Ext}$ is equal to the torsion of $H^2(\Mmfd)\cong H_1(\Mmfd,\partial\Mmfd)$ (by Lefschetz duality) since $\mathrm{Hom}$ is torsion-free. The long exact sequence in homology
$$\cdots\to H_1(\partial M)\xrightarrow{i_*} H_1(M)\xrightarrow{j_*} H_1(M,\partial M)\xrightarrow{\partial} H_0(\partial M)\to\cdots$$
implies that $H_1^{cusp}(M)=H_1(\Mmfd)/\mathrm{Im}(i_*)\cong H_1(\Mmfd)/\ker(j_*)\cong\mathrm{Im}(j_*)\cong\ker(\partial)$. But $H_0(\partial M)$ has no torsion, so the torsion of $\ker(\partial)$ matches that of $H_1(M,\partial M)$ match.
\end{proof}

This argument was given by Ian Agol.

\begin{lemma} \label{lemma:rationalHomologyLinkComplementImpliesInfinite}
The universal regular tessellation $\MunivPrinCong{\pqr}{z}$ is a homology link complement (i.e., $H_1(\MunivPrinCong{\pqr}{z})\cong \Z^c$ where $c$ is the number of cusps). If there exists a finite-volume  tessellation $\Mmfd$ with well-defined cusp modulus $z$ that is not a rational homology link complement (i.e., $H_1(\Mmfd;\Q)\not\cong\Q^c$), then $\MunivPrinCong{\pqr}{z}$ has infinite volume.
\end{lemma}

\begin{proof}
The first part follows from the fact that a universal regular tessellation has no non-trivial cuspidal covers and thus vanishing $H_1^{cusp}$. For the second part, notice that Lemma~\ref{lemma:computeCuspidalHomology} implies that $H_1^{cusp}(\Mmfd)$ has a free factor, so the cuspidal Abelian cover induced by $\rho:H_1^{cusp}(\Mmfd)\to H_1^{cusp}(\Mmfd)\otimes \Z/n$ has at least $n$ sheets and can thus be arbitrarily large but is always covered by $\MunivPrinCong{\pqr}{z}$.
\end{proof}

\begin{remark}
The converse of the first statement is wrong: $\MpslPrinCong{-3}{5+\zeta}$ is a homology link, so $H_1(\MpslPrinCong{-3}{5+\zeta})$ is generated by peripheral curves, but it is not universal and not a link complement. See discussion section for the question whether $\pi_1(\Mmfd)$ being generated by peripheral curves is sufficient for being a link complement.
\end{remark}

\section{Proof of Infinite Universal Regular Tessellation for Special Cases} \label{sec:infProofSpecial}

\begin{table}[h]
\caption{Homologies of some arithmetically defined manifolds showing that the corresponding universal regular tessellation is infinite volume.\label{table:homSelArith}}
\begin{center}
\begin{tabular}{c||c|c|c}
Case & Manifold & Cusps & $H_1$\\ \hline \hline
$\MunivPrinCong{\{3,3,6\}}{5}$ & $\MpslPrinCong{-3}{5}$ & 104 & $\Z^{117}$\\ \hline \hline
$\MunivPrinCong{\{3,4,4\}}{3+3i}$ & $\MpslPrinCong{-4}{3+3i}$ & 60 & $\Z^{65}$ \\ \hline
$\MunivPrinCong{\{3,4,4\}}{4}$  &$\MpslPrinCongBig{-4}{4}$ & 48 & $\Z^{52}$ \\ \hline
$\MunivPrinCong{\{3,4,4\}}{4+4i}$ & $\MpglPrinCong{-4}{4+4i}$ & 48 & $\Z^{72}\oplus(\Z/4)$ \\ \hline
$\MunivPrinCong{\{3,4,4\}}{5}$  &$\MpslPrinCongBig{-4}{5}$ & 144 & $\Z^{168}$\\ \hline
$\MunivPrinCong{\{3,4,4\}}{5+2i}$ & $\MpglPrinCong{-4}{5+2i}$ & 210 & $\Z^{238}\oplus(\Z/3)$ \\ \hline
$\MunivPrinCong{\{3,4,4\}}{5+3i}$ & $\MpslPrinCong{-4}{5+3i}$ & 216 & $\Z^{265} \oplus (\Z/2)^{10} \oplus (\Z/3)$  \\ \hline
$\MunivPrinCong{\{3,4,4\}}{6}$ & $\MpslPrinCong{-4}{6}$ & 120 & $\Z^{167}\oplus(\Z/2)^5$
\end{tabular}
\end{center}
\end{table}

\begin{table}[h]
\caption{Proofs that some universal regular tessellation are infinite using Gap.\label{table:gapInfProofs}}
\begin{center}
\begin{tabular}{c||c|c|c|c}
Case & $n$ & $i$ & $Q$ & Criterion \\ \hline \hline
$\MunivPrinCong{\{4,3,6\}}{2+2\zeta}$ & 3 & 1 & $\{e\}$ & inf. Abelianization \\ \hline
$\MunivPrinCong{\{4,3,6\}}{3}$ & 3 & 1 & $\{e\}$ & inf. Abelianization \\ \hline
$\MunivPrinCong{\{4,3,6\}}{3+\zeta}$ & 1 & 1 & $\mathrm{PSL}(3,3)$ & inf. Abelianization \\ \hline
$\MunivPrinCong{\{4,3,6\}}{3+2\zeta}$ & 0 & 2 & $\mathrm{PSL}(2,19)$ & Newman $p=5$ \\ \hline
$\MunivPrinCong{\{4,3,6\}}{4}$ & 0 & 8 & $A_5$ & inf. Abelianization \\ \hline
$\MunivPrinCong{\{4,3,6\}}{4+\zeta}$ & 1 & 3 & $\mathrm{PSL}(2,7)$ & inf. Abelianization \\ \hline
$\MunivPrinCong{\{4,3,6\}}{5}$ & 1 & 3 & $A_5$ & inf. Abelianization \\ \hline \hline
$\MunivPrinCong{\{5,3,6\}}{2+2\zeta}$ & 1 & 1 & $A_7$ & inf. Abelianization \\ \hline
$\MunivPrinCong{\{5,3,6\}}{3}$ & 0 & 6 & $A_5$ & inf. Abelianization \\ \hline
$\MunivPrinCong{\{5,3,6\}}{3+\zeta}$ & 1 & 26 & $\mathrm{PSL}(2,11)$ & inf. Abelianization \\ \hline
$\MunivPrinCong{\{5,3,6\}}{3+2\zeta}$ & 0 & 20 & $\{e\}$ & inf. Abelianization \\ \hline
$\MunivPrinCong{\{5,3,6\}}{4}$ & 1 & 5 & $\mathrm{PSL}(2,13)$ & inf. Abelianization \\ \hline
$\MunivPrinCong{\{5,3,6\}}{4+\zeta}$ & 1 & 10 & $\mathrm{PSL}(2, 7)$ & Newman $p=3$ \\ \hline
$\MunivPrinCong{\{5,3,6\}}{5}$ & 0 & 1 & $\mathrm{PSL}(2,25)$ & inf. Abelianization \\ \hline \hline
$\MunivPrinCong{\{3,4,4\}}{4+3i}$ & 1 & 15 & $\mathrm{PSp}(4,3)$ & inf. Abelianization \\ \hline
$\MunivPrinCong{\{3,4,4\}}{5+i}$ & 2 & 1 & $\mathrm{PSL}(2,25)$ & inf. Abelianization 
\end{tabular}
\end{center}
\end{table}

We have two different techniques to prove that a certain $\MunivPrinCong{\pqr}{z}$ is infinite. In one technique, we can construct $\MpslPrinCongBig{D}{z}$, $\MpslPrinCong{D}{z}$, or $\MpglPrinCong{D}{z}$ as in Section~\ref{sec:constructPrinCong} and use SnapPy to compute its homology (SnapPy might crash on such large manifolds unless \texttt{pari.allocatemem()} is called a couple of times in advance). By Lemma~\ref{lemma:rationalHomologyLinkComplementImpliesInfinite}, it is sufficient to show that the Betti number is larger than the number of cusps. Table~\ref{table:homSelArith} shows the results. 

The other technique uses Gap \cite{gap:gap} to prove that the group $G$ in Section~\ref{sec:allRegTessThorughGap} with presentation~\eqref{eqn:groupRep} is infinite.
Table~\ref{table:gapInfProofs} shows these proofs. Gap cannot directly show that $G$ is infinite, so we look for a subgroup $K\subset G$ such that Gap can show that $K$ is infinite either by showing that the Abelianization $K/[K;K]$ is infinite or applying the Newman criterion given in \cite{newman:InfinityCriterion} for some prime $p$ (see last column in table). This technique was suggested by Derek Holt \cite{holt:mathoverflow}. To find $K$, consider the $n$-th derived subgroup defined by $G^{(0)}=G$ and $G^{(n+1)}=[G^{(n)};G^{(n)}]$. For each subgroup $H\subset G^{(n)}$ of index $i$ and for each epimorphism of $H$ into a suitable simple group $Q$, define $K=\mathrm{Ker}(H\to Q)$. Here is an example of how the table entry for $\MunivPrinCong{\{4,3,6\}}{4+\zeta}$ would translate into Gap code proving infinity:

\begin{verbatim}
gap> G := SymmetriesUniversalRegularTessellation(4,3,6,4,1);;
gap> G1 := DerivedSubgroup(G);;
gap> for H in LowIndexSubgroupsFpGroup(G1,3) do
>        for q in GQuotients(H, PSL(2,7)) do
>            K := Kernel(q); A := AbelianInvariants(K);
>            if 0 in A then Print("infinite "); fi;
>        od;
>    od;
infinite infinite
\end{verbatim}
We have prepared a script \texttt{infiniteUniversalRegularTessellationProofs.g} that automatically verifies that the cases in Table~\ref{table:gapInfProofs} are indeed infinite.

\section{Proof of Main Theorem} \label{section:ProofMainThm}

The first statement of the main theorem, Theorem~\ref{thm:main}, is just Lemma~\ref{lemma:linkImpliesUniv}. Section~\ref{sec:pythonForUniv} described how the finite universal regular tessellations were constructed and Section~\ref{sec:howToFindDehnFillings} how Dehn fillings were found for each case not marked with a star in Table~\ref{table:finiteVolUnivReg} to verify that these universal regular tessellations are link complements. It is thus left to show that every universal regular tessellation $\MunivPrinCong{\pqr}{z}$ not listed in this table is infinite. Theorem~\ref{thm:bigMod} does this for large cusp modulus. For the remaining cases, this was done in Section~\ref{sec:infProofSpecial}. Figure~\ref{fig:OverviewArithCases} gives an overview for the cases $\{3,3,6\}$ and $\{3,4,4\}$ indicating where the generic proof in Theorem~\ref{thm:bigMod} applies and where the special cases are needed.

\section{Discussion} \label{section:dicussion}

The fundamental group being generated by peripheral curves is a necessary condition for a manifold to be a link complement. This fact is the only technique we used here to prove that a manifold is not a link complement. In particular, we did not invoke Gromov and Thurston's $2\pi$-Theorem (see \cite{BleilerHodgson:DehnFilling}), the later improvement to the 6-theorem in \cite{Lackenby:WordHypDehnSurgery} and \cite{Agol:BoundsOnExceptionalDehnFilling}, or the bound on the systole of link complements given in \cite{adamsReid:SystoleBoundLinkComplement}.

The 6-theorem states that if $M$ is a hyperbolic manifold and $s_k$ are slopes of length greater than 6 (measured on the boundary of disjoint cusp neighborhoods), then Dehn filling $M$ along these slopes results again in a hyperbolic manifold that, in particular, cannot be $S^3$.

This theorem implies many classification results about exceptional Dehn fillings, among them the software by Kazuhiro Ichihara and Hidetoshi Masai \cite{fefgen} for listing all potential exceptional slopes. This is powerful enough to decide whether a hyperbolic manifold is a knot complement, because it leaves only finitely many Dehn fillings to test for being $S^3$. But when applied to link complements, it generally leaves infinitely many Dehn fillings to test. Moreover, the techniques in Ichihara and MasaiÕs work do not scale to the size of the manifolds encountered in this paper.

Thus, we still had to use a method based on greedily minimizing volume in Section~\ref{subsubsec:findingSlopes} to find Dehn fillings resulting in $S^3$; this method is rather ad hoc, but it proved to be successful.

\subsection{Peripherally Generated Fundamental Groups and Link Complements}

Is being a peripherally generated fundamental group also a sufficient condition for being a link complement? The same question was asked for hyperbolic manifolds in \cite{bakerReid:prinCong}:
\begin{quote}
Let $M=\H^3/\Gamma$ be a finite volume orientable hyperbolic 3-manifold for which $\Gamma$ is generated by parabolic elements. Is $M$ homeomorphic to a link complement in $S^3$?
\end{quote}

It turns out that the answer is negative, a one-cusped counterexample being the manifold \texttt{m011}: we verified in SnapPy and Gap that its fundamental group is generated by peripheral curves. But according to \texttt{fef\_gen.py} (part of Ichihara and MasaiÕs work cited above), there are seven potential exceptional slopes: $(-1, 1), (-1, 2), (0, 1), (1, 0), (1, 1), (1, 2), (2, 1)$, each, $H_1\not\cong 0.$

\subsection{The unsettled case $\MunivPrinCong{\{3,4,4\}}{4+2i}$}

 Although this finite-volume manifold $\MunivPrinCong{\{3,4,4\}}{4+2i}$ fulfills the universal cuspidal cover property, this is not sufficient for being a link complement as we have seen earlier. The technique used to find the Dehn-fillings for the other cases fails in this case, suggesting that $\MunivPrinCong{\{3,4,4\}}{4+2i}$ might actually not be a link complement. However, the other theorems mentioned above (e.g., Gromov and Thurston's $2\pi$-Theorem) are not strong enough to disprove link complement in this case. So, this case unfortunately remains unsettled.

\subsection{The unsettled case $\MunivPrinCong{\{5,3,6\}}{2+\zeta}$}

 It is not even known in this case whether the manifold $\MunivPrinCong{\{5,3,6\}}{2+\zeta}$ is finite volume. Experimentally though, the situation seems clearer: The growth of the algorithm to construct the universal regular tessellation $\MunivPrinCong{\{5,3,6\}}{2+\zeta}$ is 84, 588, 3528, 17640, 79380, 353976, 1545852, 6630288, 28208124, ... simplices. Extra code furthermore confirmed that the $4$-clusters consisted of at most three hexagons for each of those surfaces $\MunivPrinCong{\{5,3,6\}}{2+\zeta}(n)$ that the computer could still construct in reasonable time. Together, this strongly suggests that $\MunivPrinCong{\{5,3,6\}}{2+\zeta}$ is infinite volume but so far we have no proof.

\section*{Acknowledgements}

I am very grateful to Ian Agol for advising me and for his many deep insights. The results in this paper would not have been possible without him. I wish to thank Christian Zickert for many fruitful discussions and William P. Thurston for a very stimulating conversation. I also want to acknowledge Derek Holt for showing me how to prove groups infinite in Gap. I also wish to thank the referee for his well-written and insightful report and for pointing out to me some recent work about exceptional Dehn fillings.

\myComment{
\section{TODO}

Number field for minimally twisted 7-component chain link is:
$x^3 - 11/5*x^2 + 8/5*x - 1/5$

Formulate main theorem correctly if we fail to prove the last finite volume manifold universal regular tessellation to be finite volume.

Add links to hyp3d videos

$g_P=\left(\begin{array}{cc} 0 & \frac{1}{2}\\ -1 & 1\end{array}\right)$ for $\{4,3,6\}$.

$g_P=\left(\begin{array}{cc} 0 & \frac{3-\sqrt{5}}{2} \\ -1 & 1 \end{array}\right)$ for $\{5,3,6\}$.

\clearpage
}

\bibliographystyle{amsalpha}
\bibliography{math_bibliography}

\end{document}